\documentclass{article}
	\usepackage{CJK,CJKnumb,CJKulem,times,dsfont,ifthen,mathrsfs,latexsym,amsfonts,color}
	\usepackage{amsmath,amsthm,makeidx,fontenc,amssymb,bm,graphicx,psfrag,listings,curves,extarrows}
	\usepackage[colorlinks,linkcolor=red,anchorcolor=blue,citecolor=blue,CJKbookmarks=True]{hyperref}
	\newtheorem{definition}{Definition}[section]
	\newtheorem{theorem}{Theorem}[section]
	\newtheorem{lemma}{Lemma}[section]
	\newtheorem{corollary}{Corollary}[section]
	\newtheorem{proposition}{Proposition}[section]
	\newtheorem{remark}{Remark}[section]
	\newcommand{\RN}{\mathbb R^N}
	\newcommand{\om}{\Omega}
	
	\newcommand{\iy}{\infty}

	\newcommand{\BR}{{\mathbb R}}
	\newcommand{\la}{\lambda}

	\newcommand{\R}{\mathbb R}

	\newcommand{\bt}{\begin{theorem}}
	\newcommand{\et}{\end{theorem}}
	\newcommand{\bl}{\begin{lemma}}
	\newcommand{\el}{\end{lemma}}
	\newcommand{\bd}{\begin{definition}}
	\newcommand{\ed}{\end{definition}}
	\newcommand{\bc}{\begin{corollary}}
	\newcommand{\ec}{\end{corollary}}
	\newcommand{\bp}{\begin{proof}}
	\newcommand{\ep}{\end{proof}}
	\newcommand{\bx}{\begin{example}}
	\newcommand{\ex}{\end{example}}
	\newcommand{\bi}{\begin{exercise}}
	\newcommand{\ei}{\end{exercise}}
	\newcommand{\bo}{\begin{proposition}}
	\newcommand{\eo}{\end{proposition}}
	\newcommand{\br}{\begin{remark}}
	\newcommand{\er}{\end{remark}}
	\newcommand{\be}{\begin{equation}}
	\newcommand{\ee}{\end{equation}}
	\newcommand{\ba}{\begin{align}}
	\newcommand{\ea}{\end{align}}
	\newcommand{\bn}{\begin{enumerate}}
	\newcommand{\en}{\end{enumerate}}
	\newcommand{\bg}{\begin{align*}}
	\newcommand{\bcs}{\begin{cases}}
	\newcommand{\ecs}{\end{cases}}

	\newcommand{\GR}{{\cal G}}

	\newcommand{\MR}{{\cal M}}
	\newcommand{\NR}{{\cal N}}
	
	\newcommand{\PR}{{\cal P}}
	\newcommand{\RR}{{\cal R}}

	\newcommand{\WR}{{\cal W}}

	\newcommand{\bean}{\begin{eqnarray*}}
	\newcommand{\eean}{\end{eqnarray*}}

	\newcommand{\vu}{\vec{\textbf{u}}}
	\newcommand{\vt}{\vec{\textbf{t}}}
	\newcommand{\hb}{{\cal B}}
	\newcommand{\bu}{ \hat{\textbf{u}} }

\numberwithin{equation}{section}

\begin{document}
\title{\bf  Positive least energy solutions for $k$-coupled Schr\"odinger system with critical exponent:\\ the higher dimension and  cooperative case\thanks{Supported by NSFC(20171301826, 20181301532).\newline
E-mails:
~yinx16@mails.tsinghua.edu.cn (Yin);~ 
~zou-wm@mail.tsinghua.edu.cn (Zou)}}

\date{}
\author{
{\bf Xin Yin,\; \;Wenming Zou}\\
\footnotesize \it Department of Mathematical Sciences, Tsinghua University, Beijing 100084, China.}

\maketitle

	\begin{center}
	\begin{minipage}{120mm}
	\begin{center}{\bf Abstract}\end{center}
	In this paper, we study the following $k$-coupled nonlinear Schr\"odinger system with Sobolev critical exponent:
	\begin{equation*}
		\left\{
		\begin{aligned}
			-\Delta u_i & +\lambda_iu_i =\mu_i u_i^{2^*-1}+\sum_{j=1,j\ne i}^{k} \beta_{ij} u_{i}^{\frac{2^*}{2}-1}u_{j}^{\frac{2^*}{2}} \quad \hbox{in}\;\Omega,\\
			u_i&>0 \quad \hbox{in}\; \Omega \quad \hbox{and}\quad u_i=0 \quad \hbox{on}\;\partial\Omega, \quad i=1,2,\cdots, k.
		\end{aligned}
		\right.
	\end{equation*}
	Here $\Omega\subset \BR^N $ is a smooth bounded domain, $2^{*}=\frac{2N}{N-2}$ is the Sobolev critical exponent, $-\la_1(\om)<\la_i<0, \mu_i>0$ ~and~$ \beta_{ij}=\beta_{ji}\ne 0$, where $\la_1(\om)$ is the first eigenvalue of $-\Delta$ with the Dirichlet boundary condition. 
 	We characterize the positive least energy solution of the $k$-coupled system for the purely cooperative case $\beta_{ij}>0$, in higher dimension $N\ge 5$.
 	Since the $k$-coupled case is much more delicated, we shall introduce {\it the idea of induction}. 
	We point out that {\it the key idea} is to give a more accurate upper bound of the least energy. 
	It's interesting to see that the least energy of the $k$-coupled system decreases as $k$ grows.
	Moreover, we establish the existence of positive least energy solution of the limit system in $\RN$, as well as classification results.
	 
	\vskip0.23in

	{\it   Key  words:}  $k$-Coupled Schr\"odinger System; Positive least energy solution; Existence; Uniqueness.


	\vskip0.23in
	\end{minipage}
	\end{center}
	\vskip0.26in
\newpage

\section{Introduction}

In this paper we consider the following $k$-coupled nonlinear Schr\"odinger system:

	\begin{equation}\label{problem}
		\left\{
		\begin{aligned}
			-\Delta u_i &+\lambda_iu_i=\mu_i u_i^{2p-1}+\sum_{j=1,j\ne i}^{k} \beta_{ij} u_{i}^{p-1}u_{j}^{p} \quad \hbox{in}\;\Omega,\\
			u_i&>0 \quad \hbox{in}\; \Omega \quad \hbox{and}\quad u_i=0 \quad \hbox{on}\;\partial\Omega, \quad i=1,2,\cdots, k,
		\end{aligned}
		\right.
	\end{equation}
where $\om= \BR^N $ or ~$\om\subset \BR^N $ is a smooth bounded domain,~$p>1$ and ~$p\le 2^*/2$ ~if $N\ge 3$,~$\mu_i>0$ ~and~$ \beta_{ij}=\beta_{ji}\ne 0$ is a coupling constant.

\vskip0.1in

System \eqref{problem} arises when we consider the standing wave solutions to the  following  time-depending  coupled nonlinear Schr\"odinger system:
\begin{equation}\label{Schrodinger original}
	\begin{cases}
	-i\frac{\partial}{\partial t}\Phi_1 =\Delta \Phi_1+\mu_1 |\Phi_1|^2\Phi_1+\beta |\Phi_2|^2\Phi_1,\\
	-i\frac{\partial}{\partial t}\Phi_2=\Delta \Phi_2+\mu_2|\Phi_2|^2\Phi_1+\beta |\Phi_1|^2\Phi_2,\\
	~\Phi_j=\Phi_j(x,t)\in \mathbb{C},~~j=1,2,\;N\leq 3,
	\end{cases}\;\;(x,t)\in \R^N\times \R,
\end{equation}
where $i$ is the imaginary unit, $\mu_1,\mu_2>0$ and $\beta\neq 0$ is a coupling constant. The system \eqref{Schrodinger original} has many applications in physics, such as the occurrence of phase separation in Bose--Einstein condensates with multiple states, or the propagation of mutuary in coherent wave packets in nonlinear optics, see e.g.
\cite{akhmediev1999partially,esry1997hartree,frantzeskakis2010dark,timmermans1998phase}.
\vskip0.1in

From the physical aspect, the solution $\Phi_j$ denotes the $j$th component of the beam in Kerr-like photorefractive media (see \cite{akhmediev1999partially}), $\mu_j$ represents self-focusing in the $j$th component, and the coupling constant $\beta$ represents the interaction between the two components of the beam.

\vskip0.1in
To study the solitary wave solutions of system \eqref{Schrodinger original}, we set $\Phi_j(x, t) = e^{i\la_jt}u_j(x)$ for $j =1, 2$.
Then, it is reduced to system \eqref{problem} with $k=2$ and $ p=2$.

\vskip0.1in
For the subcritical case, that is, $N\le 3$, the existence of least energy solutions were studied in \cite{ambrosetti2007standing,bartsch2006note,bartsch2007bound,chen2013optimal,conti2002nehari,conti2005asymptotic,dancer2010priori,Lin2005,noris2010uniform,sato2015least,sirakov2007least,terracini2009multipulse,wei2008radial} and the references therein. The  existence and multiplicity of positive and sign-changing solutions were studied by
\cite{ambrosetti2007standing,bartsch2010liouville,bartsch2007bound,chen2013multiple,chen2012infinitely,chen2013optimal,Lin2005,lin2005spikes,liu2015multiple,maia2008infinitely,noris2010existence,sato2015least,tavares2012sign} and the references therein.

\vskip0.1in
For the critical case $p = 2$ and $N=4 $ , Chen and Zou \cite{Chen2012} proved that there exists $0<\beta_1<\beta_2$~($\beta_1$ and $ \beta_2 $ depend on $ \la_i$ and $\beta$), such that \eqref{problem} has a positive least energy solution if $\beta<\beta_1 $ or $\beta>\beta_2$, while \eqref{problem} has no positive least energy solution if $\beta_1\le\beta\le \beta_2$.

\vskip0.1in
Later in \cite{Chen2015}, Chen and Zou considered the critical case $ 2p= 2^*:=\frac{2N}{N-2}$ of system \eqref{problem} in higher dimension $N \ge 5$.
It turns out that different phenomena happen from the special case $N=4$.  Indeed, the authors showed that system \eqref{problem} has a positive least energy solution for any $\beta\ne0$.
When $\beta=0$, the system is reduced to the well-known Br{\'{e}}zis-Nirenberg problem\cite{Brezis1983}
\begin{equation}\label{Brezis-Nirenberg=mui}
	-\Delta u+\la_iu=\mu_i|u|^{2^*-2}u, \quad u\in H^1_0(\Omega).
\end{equation}
The Br{\'{e}}zis-Nirenberg problem \eqref{Brezis-Nirenberg=mui} has been studied intensively, and we refer the reader to \cite{cerami1984bifurcation,chen2012brezis,devillanova2002,Schechter2010}
and references therein.

\vskip0.1in
From then on, many results corresponding to the coupled system with critical Sobolev exponent have been studied. See \cite{Liu2017,Peng2017,Wu2019,Ye2014,yang2019stable} and references therein.

\vskip0.1in
For the more general case where $k\ge2$, physically the coefficient $\beta_{ii}=\mu_i$ represents self-focusing within the $i$th component, while $\beta_{{ij}}~(i\ne j)$ illustrates interaction between two different components. 
If $\beta_{ij}>0$, the interaction between the $i$th and $j$th components is of cooperative type;
while $\beta_{ij} < 0$ means that the interaction is of competitive type.
The relation $\beta_{ij}=\beta_{ji}$ illustrates the symmetry in the interaction between different components.

\vskip0.1in
Let $\la_1(\om) $ be the first eigenvalue of $-\Delta$ with the Dirichlet boundary condition.
\vskip0.1in
For the purely cooperative case, where $\beta_{ij}>0$ for $i,j=1,2,\dots,k$, Guo, Luo and Zou \cite{Guo2018} obtained the existence and classification of positive least energy solutions to \eqref{problem} under $ -\la_1(\Omega)<\la_1=\dots=\la_k<0 $ in case of a bounded smooth domain $\Omega\subset \R^4$ with some additional conditions on the coefficients.
In the same paper,  the authors treated the system with $\la_1=\dots=\la_k=0 $ and $\Omega=\R^4$, which can be seen as the limit problem corresponding to the previous problem.
Wu \cite{Wu2019a} also obtained the existence of positive least energy solution in $\RN$ with $N\ge 4$ and $\beta_{ij}=\beta,~ i\ne j$ via variational arguments.
\vskip0.1in
Meanwhile, for the purely competitive case $\beta_{ij}<0$, \cite{Clapp2019,Tavares2020,Wu2017} studied the existence of positive least energy solutions and phase separation phenomena respectively for $N=4$ and $N\ge 5$. And some results are obtained for mixed cooperative and competitive case when $N=4$. See \cite{Tavares2020,Wei2019} and references therein.
\vskip0.1in
To the best our knowledge, there is no paper considering \eqref{problem} with arbitrary purely cooperative coefficients $\beta_{ij}>0$ in higher dimension $N\ge 5$.
It is natural to ask 
{\it whether the more general $k$-coupled system \eqref{problem} admits a positive least energy solution} similarly as the $2$-coupled system proved in \cite{Chen2015}.
In the sequel we assume that
\begin{equation}\label{2p=2ast}
	  2p=2^*.
\end{equation}
We shall see that the $k$-coupled case are more delicate than the $2$-coupled case.
{\it The key reason} is that $k\ge 2$ and $1<p<2$ when $N\ge5$.
The fact that $k\ge 2$ and $1<p<2$ make the problem more complicated comparing to the cases $p=2, k\ge 2$ and $ 1<p<2,~k=2$.
Some more ideas and techniques are needed. 

\vskip 0.2in

First, we obtained the following results.

\begin{theorem}\label{non-existence theorem of positive solutions}
	Assume that $N\ge 3$ and \eqref{2p=2ast}. Then system \eqref{problem} has no positive solutions if one of the following conditions holds.
	\begin{itemize}
		\item[(1)] $\la_i\le -\la_1(\om),~\mu_i>0 ~\hbox{and}\; \beta_{ij}>0 $~ for~ $i,j=1,2,\dots,k. $
		\item[(2)] $\mu_i, \beta_{ij}\in \R,~\la_i \ge 0 $ ~for ~$i=1,2,\dots,k,$~ and ~$\om $~ is star-shaped.
	\end{itemize}
\end{theorem}
\vskip 0.1in

If $\om=\RN$, and $ (u_1,\dots,u_k)$ is any a solution of \eqref{problem}, by the Pohozaev identity \eqref{pohozaev}, it is easy to get that 
$\int_{\RN}\sum \la_iu_i^2~\mathrm{d}x=0$, so $(u_1,\dots,u_k)\equiv(0,\dots,0) $ if $\la_1,\dots,\la_k$ are of the same sign.
Therefore, in the sequel we assume that $\om\subset\RN $ is a smooth bounded domain.
Some non-existence results also have been attained in Theorem \ref{non-existence theorem of positive solutions}. In order to study the existence of positive least energy solution of \eqref{problem}, we assume that $-\la_1(\om)<\la_i<0,~\mu_i>0 $ and $\beta_{ij}>0 $~ for~ $i,j=1,2,\dots,k$.

\vskip0.1in
We call a solution $(u_1,\dots,u_k)$ is nontrivial if $u_i\not\equiv0 $ for all $i=1,2,\dots, k$, while a solution  $(u_1,\dots,u_k)$ is semi-trivial if $u_i\equiv0 $ for some $i=1,2,\dots, k$.
We only study the nontrivial solutions of \eqref{problem} in this paper.

\vskip 0.1in

Denote $ H := (H_{0}^{1}(\om))^{k}$ and $\vu=(u_1,u_2,\dots,u_k)\in H $ is a vector function .~It is well known that solutions of \eqref{problem} correspond to the critical points of $C^1 $ functional $E:H\to \BR $ given by
\begin{equation}
	\begin{aligned}\label{energy functional E}
	    E(\vu) 
	    =&~~\dfrac{1}{2}\int\limits_{\om}\sum_{i=1}^{k}(|\nabla u_i|^2+\lambda_iu_i^2)\\
	    &-\dfrac{1}{2p}\int\limits_{\om}(\sum_{i=1}^{k}\mu_i|u_i|^{2p}+\sum_{i=1,j>i}^{k} 2\beta_{ij} |u_{i}|^{p}|u_{j}|^{p} ).
	\end{aligned}
\end{equation}
We say a solution $\vu$ of \eqref{problem} is a least energy solution, if $\vu$ is nontrivial and $$E(\vu)\le E(\vec{\textbf{v}})$$ for any other nontrivial solution $\vec{\textbf{v}}$ of \eqref{problem}. 
As in \cite{Lin2005}, we define a Nehari type manifold
	\begin{equation*}
		\MR=\{\vu\in H ~|~u_i\not\equiv 0,~E'(\vu)(0,\dots,u_i,\dots,0)=0~\hbox{for}\; i=1,2,\dots,k \}.
	\end{equation*}
Then any nontrivial solutions of \eqref{problem} belong to $\MR$. 
Take $\varphi_1,\varphi_2,\dots,\varphi_k\in C_0^\iy(\om)$ with $\varphi_i\not\equiv 0 $ and $\hbox{supp}(\varphi_i)\cap \hbox{supp}(\varphi_j)=\emptyset $ for $\forall ~i,j=1,2,\dots,k ;$~then there exists $c_1,c_2,\dots,c_k>0$ such that $(\sqrt c_1\varphi_1,\sqrt c_2\varphi_2,\dots,\sqrt c_k\varphi_k)\in \MR$, so $\MR\ne \emptyset. $
\vskip0.1in
Define the least energy
\begin{equation}\label{least energy B}
	B:=\inf\limits_{\vu\in\MR}E(\vu)
	=\inf\limits_{\vu\in\MR}\frac{1}{N}\int\limits_{\om}\sum_{i=1}^{k}(|\nabla u_i|^2+\lambda_iu_i^2).
\end{equation}
Now we consider the special case $-\la_1(\om)<\la_1=\la_2=\cdots=\la_k=\la<0$. It is well known that   the Br{\'{e}}zis-Nirenberg problem \cite{Brezis1983}
\begin{equation}\label{Brezis-Nirenberg problem 1}
	-\Delta u+\lambda u=|u|^{2^*-2}u, \quad u\in H^1_0(\om) 	
\end{equation}
has a positive least energy solution $\omega $ with energy
\begin{equation}\label{Brezis-Nirenberg least energy}
	B_1:=\frac{1}{N}\int\limits_{\om}(|\nabla \omega|^2+\lambda\omega^2)~\mathrm{d}x
	=\frac{1}{N}\int\limits_{\om} \omega^{2^*} ~\mathrm{d}x.
\end{equation}
Consider the following nonlinear algebric problem 
\begin{equation}\label{algebric system}
	\left\{
		\begin{aligned}
		    \mu_it_i^{2p-2}+\sum\limits_{j=1,j\ne i}^{k}\beta_{ij} t_i^{p-2}t_j^{p} =& 1,~~\hbox{for}\; i=1,2,\dots,k,\\
		    t_i>&0,~~\hbox{for}\; i=1,2,\dots,k.
		\end{aligned}
	\right. 
\end{equation}
In \cite{Chen2012,Chen2015}, Chen and Zou proved the exsitece of solution to \eqref{algebric system} in the case $k=2$.
Recently, Wu \cite{Wu2019a} proved that \eqref{algebric system} has a solution for $k\ge2$ with $\beta_{ij}=\beta,$ by applying the variational argument to the related system.
We generalize the existence result of the algebric system \eqref{algebric system} to arbitrary $\beta_{ij}>0$ and give a more minute and geometric proof in Lemma \ref{existence algebric}.
Moreover, we will prove in Lemma \ref{existence algebric} that there exists $\vec{\tilde{\textbf{t}}}=(\tilde{t}_1,\tilde{t}_2,\ldots,\tilde{t}_k)$ such that
\begin{equation}\label{least algebric solution}
	\vec{\tilde{\textbf{t}}}~~\hbox{satisfies}~\eqref{algebric system}~\hbox{and}~\vec{\tilde{\textbf{t}}}~\hbox{attains}~d_{k}.
\end{equation}
Here, $d_k:=\inf\limits_{\vt\in\PR_k }\GR_k(\vt)$, where
	\begin{equation*}
		\GR_k(\vt):=\sum\limits_{i=1}^k(\frac{t_i^2}{2}-\frac{\mu_i|t_i|^{2p}}{2p})- \frac{1}{2p}\big(\sum\limits_{i=1,j>i}^{k}2\beta_{ij}|t_i|^{p}|t_j|^{p} \big),
	\end{equation*}
	\begin{equation*}
		\PR_k:= \{\vt\in\R^k \backslash\{\vec{\textbf{0} } \} ~|~P_k(\vt):=\sum\limits_{i=1}^k(t_i^2-\mu_i|t_i|^{2p})-  \sum\limits_{i=1,j>i}^{k}2\beta_{ij}|t_i|^{p}|t_j|^{p}=0   \}.
	\end{equation*}
We will see that $d_{k}$ has a similar form as the minimizing problem \eqref{A} and $d_k$ plays an important role in this paper.

\vskip 0.1in

Our second result also  deals with the special case $\la_1=\la_2=\cdots=\la_k:=\la$.


\begin{theorem}\label{symmetric case}
	Assume that $N\ge 5$ and \eqref{2p=2ast},      $-\la_1(\Omega)<\la_1=\la_2=\cdots=\la_k:=\la <0$ and $\mu_i,\beta_{ij}>0$.
	Then system \eqref{problem} has a positive solution of the synchronized form
		\begin{equation}\label{synchronized form}
				(t_1\omega,t_2\omega,\dots,t_k\omega),
		\end{equation}
	where $\omega$ is the positive least energy solution of \eqref{Brezis-Nirenberg problem 1} and $(t_1,t_2,\dots,t_k)$ solves \eqref{algebric system}.
	
\end{theorem}

\begin{remark}  \quad

	\begin{itemize}
		\item[(1)] Denote $\RR =(\beta_{ij})_{k\times k}$ as the coupling matrix, where $\beta_{ij}=\beta_{ji} $  and $ \beta_{ii}=\mu_i$.
		For the special case $N=4$,~Guo, Luo and Zou \cite{Guo2018} proved that the least energy $B$ is attained by a symmetrized form solution \eqref{synchronized form},~under the assumptions that
		\begin{equation}\label{matrix condition}
			\RR=(\beta_{ij})_{k\times k}~~\hbox{is invertible and}~~\sum_{j=1}^{k}(\RR^{-1})_{ij}>0,\quad i=1,2,\dots,k.
		\end{equation}
		Actually, under these assumptions, the algebric system has pairs of solutions $(\pm t_1,\pm t_2,\dots,\pm t_k)$, then the signed solutions $(\pm t_1\omega,\dots,\pm t_k\omega)$ are nontrivial signed solutions of \eqref{problem}.

		\item[(2)] If $\la\le -\la_1(\om)$, we can obtain sign-changing solutions of synchronized form.
	\end{itemize}
\end{remark}

We shall give a classification result of the positive least energy solution of the \eqref{problem}.

\begin{theorem}\label{symmetric theorem existence and uniqueness}
	Assume that $N\ge 5$, $-\la_1(\Omega)<\la_1=\la_2=\cdots=\la_k:=\la <0$ and $\mu_i,\beta_{ij}>0$.
	Then the positive least energy solution of the system \eqref{problem} must be of the synchronized form
		\begin{equation*}
				(\tilde{t}_1\omega,\tilde{t}_2\omega,\dots,\tilde{t}_k\omega),
		\end{equation*}
	where $\omega$ is a positive least energy solution of \eqref{Brezis-Nirenberg problem 1} and $(\tilde{t}_1,\tilde{t}_2,\dots,\tilde{t}_k)$ is defined in \eqref{least algebric solution}.

	Moreover, assume additionally that $\beta_{ij}:=\beta>0$ and $\Omega\subset \RN $ is a ball, then there exists $\beta_{k}>0$, such that the positive least energy solution of the system \eqref{problem} is unique for $\beta>\beta_k$.
\end{theorem}

\begin{remark}\quad
	\begin{itemize}
	 	\item[(1)] For the special case $N=4$,~Guo, Luo and Zou \cite{Guo2018} classified the positive least energy solution of the synchronized form \eqref{synchronized form},~under the assumptions \eqref{matrix condition} and $\RR$ is positively or negatively definite.
	 	In particular, the positive least energy solution is unique if $\Omega$ is a ball in $\mathbb{R}^4$.
		\item[(2)] Recently, Tavares and You \cite{Tavares2020} also obtained the classification result for the special case $N=4$, under the same assumption as our result. But the uniqueness problem was not covered in \cite{Tavares2020}.
	 \end{itemize} 
\end{remark}

\vskip 0.1in

Now let us consider the general case $-\la_1(\om)<\la_1,\la_2,\cdots,\la_k<0$.

\begin{theorem}\label{main theorem}
Assume that $N\ge 5$ and \eqref{2p=2ast}. 	 
Suppose that $-\la_1(\om)<\la_1,\la_2,\dots,\la_k<0$ and $\mu_i>0$.~Then system \eqref{problem} has a positive least energy solution $\vu$ with $E(\vu)=B$ for any $\beta_{ij}>0$.
\end{theorem}

\begin{remark}  In \cite{Clapp2019,Wu2017}, the authors obtained the existence of positive least energy solutions of the system \eqref{problem} for the purely competitive case $\beta_{ij}<0$.  \end{remark} 

\vskip0.1in


For each $\Theta\subsetneqq \{1, . . . ,k\}$,    we obtain the following $(k -|\Theta |)$-coupled system
	\begin{equation}\label{problem k Abstract}
		\left\{
		\begin{aligned}
			-\Delta u_i & +\lambda_iu_i=\mu_i u_i^{2p-1}+\sum_{j=1,j\ne i}^{k} \beta_{ij} u_{i}^{p-1}u_{j}^{p} \quad \hbox{in}\;\Omega,\\
			u_i&>0~~ \hbox{in}\; \Omega \quad \hbox{and}\quad u_i=0~~ \hbox{on}\;\partial\Omega, \quad i,j\in \{1,2,\cdots, k\}\backslash \Theta,
		\end{aligned}
		\right.
	\end{equation}
by replacing $\la_i, \mu_i,\beta_{ij},\beta_{ji}$ with 0 if $i \in \Theta $, where $ |\Theta| $ is the cardinality of $\Theta $. 

\vskip0.1in
Note that the nontrivial solutions of \eqref{problem k Abstract} can be extended to the semi-trivial solutions of $k$-coupled system \eqref{problem}.
\vskip0.1in
Denote $\tau=k-|\Theta|$. Indeed, we can establish $C_{k}^{\tau}$ different $\tau$-coupled systems from \eqref{problem}.
For every fixed $\Theta$, we denote $\{i_1,\dots,i_{\tau}\}=\{1,2,\dots,k\}\backslash \Theta$ and $ B_{\mu_{i_1}\dots\mu_{i_\tau}} $ as the least energy of the $\tau $-coupled system \eqref{problem k} with coefficients $(\mu_{i_1},\dots,\mu_{i_\tau})$, that is,
\begin{equation}
	B_{\mu_{i_1}\dots\mu_{i_\tau}}:=\inf\{E(\vu)~:~\vu=(u_1,\dots,u_k)~\hbox{solves}~\eqref{problem}~~\hbox{and}~u_i=0~\hbox{iff}~i\in \Theta \}.
\end{equation}
Denote
\begin{equation}
	\bar{ B}_\tau :=	\min\limits_{\forall \{l_1,\dots,l_{\tau}\}\subsetneqq\{1,2,\dots,k\} } \{B_{\mu_{l_1}\dots\mu_{l_\tau}} \}.
\end{equation}
Next we  need to obtain a more refined estimates on  the least energy of the $k$-coupled system.
 \begin{theorem}\label{energy decreases}
	Assume that $N\ge 5$ and \eqref{2p=2ast}.  Suppose  that $-\la_1(\om)<\la_1,\la_2,\dots,\la_k<0$ and $\mu_{i},\beta_{ij}>0$.
	Then the least energy of \eqref{problem} has an upper bound
	\begin{equation}
		B<\min\{\bar{ B}_1,\bar{ B}_2,\dots,\bar{ B}_{k-1}\}.
	\end{equation}
	To be more specific,
	\begin{equation*}
		B<\bar{ B}_{k-1}<\dots<\bar{ B}_1.
	\end{equation*}
	Moreover,
	\begin{equation*}
		B:=B_{\mu_1\mu_2\dots\mu_k}<B_{\mu_1\mu_2\dots\mu_{k-1}}
		<B_{\mu_1\mu_2\dots\mu_{k-2}}<\dots<B_{\mu_1\mu_2}<B_{\mu_1}.
	\end{equation*}
\end{theorem}

By Theorem \ref{energy decreases}, it's interesting to see that the least energy of the $k$-coupled system may decrease as $k$ grows.\\

Since the nonlinearity and the coupling term are both critical in \eqref{problem}, the existence of nontrivial solutions of \eqref{problem} depends heavily on the existence of the least energy solution of the following limit problem
	\begin{equation}\label{RN system}
		\left\{
		\begin{aligned}
			-\Delta u_i&=\mu_i u^{2p-1}+\sum_{j=1,j\ne i}^{k} \beta_{ij} u_{i}^{p-1}u_{j}^{p}, \quad x\in\R^N,\\
			u_i&\in D^{1,2}(\R^N),\quad i=1,2,\cdots, k,
		\end{aligned}
		\right.
	\end{equation}
where $ D^{1,2}(\R^N):=\{u\in L^{2^*}(\R^N)~:~|\nabla u|\in L^{2}({\R^N})  \} $~with norm $|| u||_{D^{1,2}}:=(\int_{\R^N}|\nabla u|^2 \mathrm{d}x )^{1/2} $.
~Let $S$ be the sharp constant of $D^{1,2}(\R^N)\hookrightarrow  L^{2^*}(\R^N)$
	\begin{equation}\label{sobolev embedding}
		\int\limits_{\R^N}|\nabla u|^2 \mathrm{d}x \ge S~\bigg(\int\limits_{\R^N}|u|^{2^*} \mathrm{d}x\bigg )^{\frac{2}{2^*}}.
	\end{equation}
		For $ \varepsilon>0$ and $y\in \R^N $, we consider the Aubin-Talenti instanton \cite{Aubin1976,willem1997minimax}
	$U_{\varepsilon,y} \in D^{1,2}(\R^N) $ defined by 
	\begin{equation}\label{Aubin-Talenti}
		U_{\varepsilon,y}(x):=[N(N-2)]^{\frac{N-2}{4}}\bigg( \dfrac{\varepsilon}{\varepsilon^2+|x-y|^2 } \bigg)^{\frac{N-2}{2}}.
	\end{equation}
	Then $U_{\varepsilon,y}$ satisfies $-\Delta u=|u|^{2^*-2}u $ in $\R^{N}$ and
	\begin{equation}\label{estimate of Aubin-Talenti}
		\int\limits_{\R^N}|\nabla U_{\varepsilon,y}|^2 \mathrm{d}x = \int\limits_{\R^N}|U_{\varepsilon,y}|^{2^*} \mathrm{d}x=S^{N/2}.
	\end{equation}
Clearly \eqref{RN system} has semi-trivial solutions 
	$$(\mu_1^{\frac{2-N}{4} }U_{\varepsilon,y},0,\dots,0),~(0,\mu_2^{\frac{2-N}{4} }U_{\varepsilon,y},0,\dots,0),~\dots~,(0,\dots,0,\mu_k^{\frac{2-N}{4} }U_{\varepsilon,y}) .$$
	Here, we are only concerned with nontrivial solutions of \eqref{RN system}.
	Define $D:=(D^{1,2}(\R^N) )^k $ and a $C^1 $ functional $I:D\to \R $ given by
	\begin{equation}\label{energy functional I}
		I(\vu):=\dfrac{1}{2}\int\limits_{\R^N}\sum_{i=1}^{k}|\nabla u_i|^2
	    -\dfrac{1}{2p}\int\limits_{\R^N}(\sum_{i=1}^{k}\mu_i|u_i|^{2p}+\sum_{i=1,j>i}^{k} 2\beta_{ij} |u_{i}|^{p}|u_{j}|^{p} ).
	\end{equation}
We consider the following set as in \cite{Lin2005}:
	\begin{equation}\label{set N}
		\NR=\{\vu \in D	~:~ u_i\not\equiv 0,I'(\vu)(0,\dots,0,u_i,0,\dots,0)=0,~~i=1,2,\dots,k \}.
	\end{equation}
	Then any nontrivial solution of \eqref{RN system} belongs to $\NR$. Similarly $\NR\ne \emptyset .$
	We set 
	\begin{equation}\label{A}
		A:=\inf\limits_{\vu\in \NR}I(\vu)
		=\inf\limits_{\vu\in \NR}\frac{1}{N}\int\limits_{\R^N}\sum_{i=1}^{k}|\nabla u_i|^2.
	\end{equation}
	Then we have the following theorem, which plays an important role in the proof of Theorem \ref{main theorem}.

\begin{theorem}\label{main theorem_limit}
	Assume that $N\ge 5$ and \eqref{2p=2ast}.  If $\mu_i>0$ and $\beta_{ij}>0$,~then system \eqref{RN system} has a positive least energy solution $\vec{\textbf{U}}$ with $I(\vec{\textbf{U}})=A$,~which is radially symmetric decresing. Moreover,
	\begin{itemize}
		\item[(1)] the positive least energy solution of \eqref{RN system} must be the least energy synchronized type solution of the form
		$(\tilde{t}_1U_{\varepsilon,y},\tilde{t}_2U_{\varepsilon,y},\dots,\tilde{t}_kU_{\varepsilon,y})$, where $(\tilde{t}_1, \tilde{t}_2,\dots,\tilde{t}_k)$ is defined in \eqref{least algebric solution}.

		\item[(2)] there exists $\beta_1>0$, and for any $0<\beta_{ij}=\beta<\beta_1$, there exists a solution $(t_1(\beta),t_2(\beta),\dots,t_k(\beta))$ of \eqref{algebric system}, such that
		$$  I(t_1(\beta)U_{\varepsilon,y},t_2(\beta)U_{\varepsilon,y},\dots,t_k(\beta)U_{\varepsilon,y})>A=I(\vec{\textbf{U} } ) . $$
		That is,~$(t_1(\beta)U_{\varepsilon,y},t_2(\beta)U_{\varepsilon,y},\dots,t_k(\beta)U_{\varepsilon,y}) $ is a different positive solution of \eqref{RN system} with respect to $ \vec{\textbf{U} }$.
	\end{itemize}
\end{theorem}

\begin{remark}\quad

	\begin{itemize}
		\item[(1)] The multiplicity of solution of algebric system \eqref{algebric system} arises much more difficulty when dealing with the classification problem.

		For the special case $\beta_{ij}=\beta>0$, Wu proved that the positive least energy solution of \eqref{RN system} must be of synchronized type $(t_1U_{\varepsilon,y},t_2U_{\varepsilon,y},\dots,t_kU_{\varepsilon,y})$ where $(t_1, t_2,\ldots,t_k)$ is solves \eqref{algebric system}. (See \cite[Theorem 1.2, Proposition 3.2]{Wu2019a}.)
		
		We point out that the solution $(t_1,t_2,\ldots,t_k)$ of \eqref{algebric system} might not attain $d_{k}$, then $(t_1U_{\varepsilon,y},t_2U_{\varepsilon,y},\dots,t_kU_{\varepsilon,y})$ might not be of the least energy, since the method of Lagrange multiplier yields only a necessary condition for constrained problems. 

		It is interesting that Theorem \ref{main theorem_limit}~(2) coincides with the multiplicity of solution of algebric system \eqref{algebric system}, and admits the existence of another positive solution of \eqref{RN system} which is not of the least energy.

		\item[(2)] Theorem \ref{main theorem_limit}~(1) generalizes 
		\cite[Theorem 1.6]{Chen2015} in the sense that it asserts that the positive least energy of \eqref{RN system} must be the synchronized type solution for the purely cooperative case $\beta_{ij}>0 $.

		\item[(3)] We note that a similar existence result of ground state (which seems probably to be semi-trivial) is obtained by He and Yang \cite{he2018quantitative} via concentration-compactness lemma.
		We remark that our method is quite different from that.
		We illustrate the positive least energy of synchronized type and any other result concerning \eqref{RN system} cannot be found in \cite{he2018quantitative}. 


	\end{itemize}
\end{remark}

In rest of this paper we prove these theorems. 
In Sect. \ref{section 2}, we obtain the fundamental Lemma \ref{existence algebric} on the existence of positive solution of the algebric system \eqref{algebric system} with arbitrary $\beta_{ij}>0 $, which generalizes \cite[Proposition 3.2]{Wu2019a} and gives out a detailed and geometric explanation. Theorem \ref{non-existence theorem of positive solutions} and \ref{symmetric case} are proved subsequently.
Since the $k$-coupled case is much more delicated than $2$-couple case, we shall introduce {\it the idea of induction} in the proof of Theorem \ref{main theorem} and \ref{main theorem_limit}, respectively in Sect. \ref{section 3} and \ref{section 4}.
{\it The key idea} is to give a more accurate upper bound of the least energy, which is inspired by \cite{Chen2015,Lin2005}.
See Lemma \ref{main lemma 1} and \ref{key lemma}. 
We mention that the idea of induction is also introduced in \cite{Clapp2019,Tavares2020,Wu2017} to obtain the positive least energy solution of the purely competitive system. 
Theorem \ref{energy decreases} is a byproduct of Lemma \ref{key lemma} and Theorem \ref{main theorem}.
Theorem \ref{symmetric theorem existence and uniqueness} is proved in Sect. \ref{section 4}.

\section{Preliminaries}\label{section 2}

In this section, we introduce some fundamental results and prove Theorem \ref{non-existence theorem of positive solutions} and Theorem \ref{symmetric case}.

\begin{lemma}\label{pohozaev lemma}(the Pohozaev identity)

	 If ~$\om $ is star-shaped with respect to some point
	  $ y_0 \in \R^N $ and $ \vu\in H $ be a solution of \eqref{problem},~then $\vu$ satisfies
	  \begin{equation}\label{pohozaev}
		  \begin{aligned}
		    &\frac{1}{2p} \int\limits_{\om}\sum_{i=1}^{k}|\nabla u_i|^2~\mathrm{d}x
		    +\frac{1}{2N}\int\limits_{\partial \om}\sum_{i=1}^{k}\bigg|\frac{\partial u_i}{\partial \vec{\textbf{n}}} \bigg|^2(x-y_0)\cdot \vec{ \textbf{n}}~\mathrm{d}S  \\
		    &=
		    \frac{1}{2p}\int\limits_{\om}(\sum_{i=1}^{k}\mu_i|u_i|^{2p}+\sum_{i=1,j>i}^{k} 2\beta_{ij} |u_{i}|^{p}|u_{j}|^{p} )~\mathrm{d}x
		    -\frac{1}{2}\int\limits_{\om}\sum_{i=1}^{k}\la_iu_i^2~\mathrm{d}x,
		  \end{aligned}
	  \end{equation}
	  where $ \vec{\textbf{n}}$ denotes the unit outward normal.
\end{lemma}

\begin{proof}
	The Pohozaev identity \eqref{pohozaev} follows by multiplying $i$th equation
	$$-\Delta u_i+\lambda_iu_i=\mu_i u_i^{2p-1}+\sum_{j=1,j\ne i}^{k} \beta_{ij} u_{i}^{p-1}u_{j}^{p} $$
	by $  x\cdot \nabla u_i $ and integrating by parts.
\end{proof}

Fisrt we prove Theorm \ref{non-existence theorem of positive solutions}, which shows the nonexistence of positive solution of system \eqref{problem}  under some assumptions.\\

{\it{Proof of Theorem \ref{non-existence theorem of positive solutions} }}
	(1)  Let $\phi_1 $ be the first eigenfunction of $-\Delta$ with the Dirichlet boundary condition with respect to $\la_1(\om) $.
 	
 	By multiplying each equations in \eqref{problem} with $\phi_1 $  and integrating over $\om $, we obtain
 	$$   (\la_1(\om)+\la_i )\int\limits_{\om} u_i\phi_1 \mathrm{d}x
 		=\int\limits_{\om} (\mu_i u_i^{2p-1}\phi_1+\sum_{j=1,j\ne i}^{k} \beta_{ij} u_{i}^{p-1}u_{j}^{p}\phi_1) ~\mathrm{d}x . $$
 	It is easy to see that when $\la_i\le -\la_1(\om),~\mu_i>0 ~\hbox{and}\; \beta_{ij}>0 $~ for~ $i,j=1,2,\dots,k, $ system \eqref{problem} has no positive solutions.

 	(2) Assume that ~$\om $ is star-shaped with respect to some point{}
	  $ y_0 \in \R^N $ and $ \vu=(u_1,u_2,\dots,u_k) $ be a solution of \eqref{problem}, then $\vu $ satisfies
	  \begin{equation*}
	  	 \int\limits_{\om}\sum_{i=1}^{k}(|\nabla u_i|^2+\la_iu_i^2)~\mathrm{d}x
		  =\int\limits_{\om}(\sum_{i=1}^{k}\mu_i|u_i|^{2p}+\sum_{i,j=1,i< j}^{k} 2\beta_{ij} |u_{i}|^{p}|u_{j}|^{p} )~\mathrm{d}x.
	  \end{equation*}
 If $\la_i\ge0 $, combined with the Pohozaev identity \eqref{pohozaev} we have
	  \begin{equation*}
	  	0\le \frac{1}{2N}\int\limits_{\partial \om}\sum_{i=1}^{k} \bigg|\frac{\partial u_i}{\partial \vec{ \textbf{n} }} \bigg|^2(x-y_0)\cdot \vec{ \textbf{n}}~\mathrm{d}S
	  	=(\frac{1}{2p}-\frac{1}{2}) \int\limits_{\om}\sum_{i=1}^{k}\la_iu_i^2~\mathrm{d}x \le0,
	  \end{equation*}
	  then system \eqref{problem} has no positive solutions.
	  \hfill$\Box$

\vskip 0.2in

Now we turn to study the existence of solution to the algebric system \eqref{algebric system}, which may play a fundamental role in the subsequent research.
We generalize \cite[Proposition 3.2]{Wu2019a} to arbitrary $\beta_{ij}>0$. Our proof is more specific and gives out a geometric explanation.

\begin{lemma}\label{existence algebric}

	Assume $\mu_i>0 $ and $ \beta_{ij}>0$, then algebric system 
		\begin{equation*}
			\left\{
				\begin{aligned}
				    \mu_it_i^{2p-1}+\sum\limits_{j=1,j\ne i}^{k}\beta_{ij} t_i^{p-1}t_j^{p} = t_i,~~&\hbox{for}\; i=1,2,\dots,k,\\
				    t_i\in~\R,~~&\hbox{for}\; i=1,2,\dots,k
				\end{aligned}
			\right. 
		\end{equation*}
	has a nontrivial positive solution.
\end{lemma}

\begin{proof}
	Let 
	\begin{equation}\label{dk}
		d_k:=\inf\limits_{\vt\in\PR_k }\GR_k(\vt),
	\end{equation}
	where
	\begin{equation*}
		\GR_k(\vt)=\sum\limits_{i=1}^k(\frac{t_i^2}{2}-\frac{\mu_i|t_i|^{2p}}{2p})- \frac{1}{2p}\big(\sum\limits_{i=1,j>i}^{k}2\beta_{ij}|t_i|^{p}|t_j|^{p} \big),
	\end{equation*}
	\begin{equation*}
		\PR_k= \{\vt\in\R^k \backslash\{\vec{\textbf{0} } \} ~|~P_k(\vt):=\sum\limits_{i=1}^k(t_i^2-\mu_i|t_i|^{2p})-  \sum\limits_{i=1,j>i}^{k}2\beta_{ij}|t_i|^{p}|t_j|^{p}=0   \}.
	\end{equation*}
Denote $ \vec{\textbf{e}_i }=(0,\dots,0,\mu_i^{-\frac{1}{2p-2}},0,\dots,0) $,
it is easy to check that $\vec{\textbf{e}_i }\in \PR_k $ for all $i=1,2,\dots,k,$ which implies $\PR_k \ne \emptyset $ and the minimizing problem \eqref{dk} is well defined. Notice that
	\begin{equation*}
		\begin{aligned}
		    d_k=\inf\limits_{\vt\in\PR_k }\GR_k(\vt)
		    &=\inf\limits_{\vt\in\PR_k }\frac{1}{N}\sum_{i=1}^{k}t_i^2\\
			&=\inf\limits_{\vt\in\PR_k }\frac{1}{N} \bigg(\mathrm{dist}~(\vec{\textbf{0}},~\vt ) \bigg)^2.
		\end{aligned}
	\end{equation*}
Since $\PR_k\ne \emptyset $ and symmetric, the minimizing problem \eqref{dk} is reduced to a geometric problem, that is, to find the closest point to original point $\vec{\textbf{0}}$ in a nonempty set $\PR_k$. 
Then $d_k $ can be attained by some $\vec{\tilde{\textbf{t} }}  $ with $ \tilde{t}_{i}\ge0 $ for all $i=1,2,\dots,k,$ and $ \tilde{t}_{i}> 0 $ for some $i$.

By the method of Lagrange's multiplier, there exists a Lagrange multiplier $\gamma\in \R $ such that
	\begin{equation}\label{Lagrange}
		\nabla \GR_k(\vec{\tilde{\textbf{t}}}) -\gamma\nabla P_k(\vec{\tilde{\textbf{t}}}) =0,
	\end{equation}
then $\vec{\tilde{\textbf{t}}} $ satisfies the system
		\begin{equation}\label{original algebric system}
			\left\{
				\begin{aligned}
				    \mu_i \tilde{t}_i^{2p-1}+\sum\limits_{j=1,j\ne i}^{k}\beta_{ij} \tilde{t}_{i}^{p-1}\tilde{t}_j^{p} =\tilde{t}_{i} ,~~&\hbox{for}\; i=1,2,\dots,k,\\
				    \tilde{t}_{i}\ge 0,\quad \hbox{and}\quad \sum\limits_{i=1}^k\tilde{t}_{i}>0 ,~~&\hbox{for}\; i=1,2,\dots,k.
				\end{aligned}
			\right. 
		\end{equation}
In fact, for any nonzero $\vt\in \PR_{k}$, we assume that $\vt=(t_1,\ldots,t_k) $ with $t_{s}\ne 0$ for $ 1\le s\le k$, we see that
\begin{equation*}
	   t_i\partial_i P_k(\vt)=2t_i^2-2p\mu_it_i^{2p}-p\sum_{j=1,j\ne i}^{k}2\beta_{ij}t_i^p t_j^p,
\end{equation*}
then we have
\begin{equation*}
	\begin{aligned}
   	 \sum_{i=1}^s t_i\partial_i P_k(\vt)
   	 & = 2\sum_{i=1}^{s} t_i^2-2p\sum_{i=1}^s \mu_it_i^{2p}-p\sum_{i=1}^{s}\sum_{j=1,j\ne i}^{k}2\beta_{ij}t_i^p t_j^p\\
   	 & = 2\sum_{i=1}^{s} t_i^2-2p\sum_{i=1}^s \mu_it_i^{2p}-2p\sum_{i=1,j>i}^{s}2\beta_{ij}t_i^p t_j^p\\
   	 & =(2-2p)\sum_{i=1}^{s}t_i^2 \ne 0,
	\end{aligned}	   
\end{equation*}
which implies that if~$\vt\in\PR_k$ then $\nabla P_k(\vt)\ne \vec{\textbf{0}}$.~
Then by \eqref{Lagrange}, we have
	\begin{equation*}
		\partial_i\GR_k(\vec{\tilde{\textbf{t}}}) -\gamma\partial_i P_k(\vec{\tilde{\textbf{t}}}) =0.
	\end{equation*}
Thus $\partial_i\GR_k(\vec{\tilde{\textbf{t}}})=0$ if $\tilde{t}_i=0$, 
while
	\begin{equation*}
		0=P_k(\vec{\tilde{\textbf{t}}})=\sum_{i=1}^{s} \tilde{t}_i\partial_i\GR_k(\vec{\tilde{\textbf{t}}}) 
		=\gamma\sum_{i=1}^{s}\tilde{t}_i\partial_i P_k(\vec{\tilde{\textbf{t}}}),
	\end{equation*}
if~$\tilde{t}_1,\dots,\tilde{t}_s \ne 0$, that is
\begin{equation*}
	\begin{aligned}
	    (2-2p)\gamma\sum_{i=1}^{s}\tilde{t}_i^2=0,
	\end{aligned}
\end{equation*}
which implies that $\gamma=0$ and $\partial_i\GR_k(\vec{\tilde{\textbf{t}}})=0$ for $\tilde{t}_i\ne 0$. 
Hence, $\vec{\tilde{\textbf{t}}} $ satisfies \eqref{original algebric system}.
\vskip 0.1in

We claim that if $\beta_{{ij}}>0$, then $d_k<\min\{d_1,d_2,\dots,d_{k-1}\} $.

Our proof is inspired by \cite{abdellaoui2009}.
By the method of induction, it suffices to prove that $d_{k}<d_{k-1}$.
Assume that $d_{k-1}$ is attained by $(t_1,\dots,t_{k-1})$, then
$$d_{k-1}=\frac{1}{N}\sum\limits_{i=1}^{k-1}t_i^2=\frac{1}{N}(\sum\limits_{i=1}^{k-1} \mu_i|t_i|^{2p}+\sum\limits_{i=1,j>i}^{k-1}2\beta_{ij}|t_i|^{p}|t_j|^{p} ). $$
Note that for any $s\in \BR$, there exists $f(s)>0$ such that $f(s)(t_1,\dots,t_{k-1},sl)\in \PR_k$, where $l\in \BR\backslash\{0\}$.
In fact,
\begin{equation*}
	f(s)^{2p-2}=\frac{Nd_{k-1}+s^2l^2}{Nd_{k-1}+\mu_k|sl|^{2p}+\sum\limits_{i=1}^{k-1}2\beta_{ik}|t_i|^p|sl|^p}.
\end{equation*}
By direct computations we have
\begin{equation*}
	\lim_{s\to0}\frac{f'(s)}{|s|^{p-2}s}
	=-\frac{p}{(p-1)Nd_{k-1}}\sum_{i=1}^{k-1} \beta_{ik} |t_{i}|^{p}|l|^{p},
\end{equation*}
that is,
\begin{equation*}
	f'(s)=-\frac{p}{(p-1)Nd_{k-1}}\bigg(\sum_{i=1}^{k-1}\beta_{ik} |t_{i}|^{p}|l|^{p}\bigg)~
	|s|^{p-2}s~(1+o(1)),\quad \hbox{as}\;~s\to 0.
\end{equation*}
Note that $f(0)=1$, then
\begin{equation*}
	f(s)=1-\frac{1}{(p-1)Nd_{k-1}}\bigg(\sum_{i=1}^{k-1}\beta_{ik} |t_{i}|^{p}|l|^{p}\bigg)~
	|s|^{p}~(1+o(1)),\quad \hbox{as}\;~s\to 0.
\end{equation*}
This implies that
\begin{equation*}
	\begin{aligned}
	    f(s)^{2p}
	    & = 1-\frac{2p}{(p-1)Nd_{k-1}}\bigg(\sum_{i=1}^{k-1} \beta_{ik} |t_{i}|^{p}|l|^{p}\bigg)~|s|^{p}~(1+o(1)) \\
		& = 1-\frac{1}{d_{k-1}}
		\bigg(\sum_{i=1}^{k-1} \beta_{ik} |t_{i}|^{p}|l|^{p}\bigg)~
		|s|^{p}~(1+o(1)) \quad \hbox{as}\;~s\to 0,
	\end{aligned}
\end{equation*}
then we see that
\begin{equation*}
	\begin{aligned}
	    d_k
	    & \le \GR_k(f(s)t_1,\dots,f(s)t_{k-1},f(s)sl )\\
	    & = \frac{1}{N}f(s)^{2p} \bigg(Nd_{k-1}
	    + \mu_k |sl|^{2p} +\sum_{i=1}^{k-1} 2\beta_{ik}|t_{i}|^{p}|sl|^{p}\bigg)\\
	    & = d_{k-1}-(\frac{1}{2}-\frac{1}{N})\bigg(\sum_{i=1}^{k-1} 2\beta_{ik}|t_{i}|^{p}|l|^{p}\bigg)|s|^{p}+o(|s|^{p})\\
	    & <d_{k-1}\quad \hbox{as}~|s|>0 ~\hbox{small enough}.
	\end{aligned}
\end{equation*}
By the idea of induction, we see that $d_{k}<\min\{d_1,d_2,\dots,d_{k-1}\}$.
Next we prove that if $\beta_{ij}>0$, then $d_{k}$ is attained by $(\tilde{t}_{1},\tilde{t}_{2},\dots,\tilde{t}_{k})$, where $\tilde{t}_{i}>0 $ for all $i=1,2,\dots,k$.
Without loss of genrality, we assume by contradiction that $\tilde{t}_{1},\tilde{t}_{2},\dots,\tilde{t}_{k-1}>0,$ and $ \tilde{t}_{k}=0.$ Denote 
	\begin{equation*}
		\vec{\tilde{\textbf{t}}}_{k}=(\tilde{t}_{1},\tilde{t}_{2},\dots,\tilde{t}_{k-1},0),	\quad
		\vec{\tilde{\textbf{t}}}_{k-1}=(\tilde{t}_{1},\tilde{t}_{2},\dots,\tilde{t}_{k-1}),
	\end{equation*}
then $\vec{\tilde{\textbf{t}}}_{k-1}\in\PR_{k-1} $.
Notice that
	\begin{equation}
		\begin{aligned}
		    d_k=\inf\limits_{\vt\in\PR_k }\GR_k(\vt)
		    &=\GR_k(\vec{\tilde{\textbf{t}}}_{k})\\
		    &=\GR_{k-1}(\vec{\tilde{\textbf{t}}}_{k-1})
		    \ge~ \inf\limits_{\vt\in\PR_{k-1} }\GR_{k-1}(\vt)=d_{k-1},
		\end{aligned}
	\end{equation} 
while we see that 
			$$d_k <\min \{ d_1, d_2, \dots, d_{k-1} \}, $$
a contradiction. Hence we conclude that all $\tilde{t}_{i}>0 $.
\end{proof}

\vskip 0.2in

Now assume that $N\ge 4$, $-\la_1(\Omega)<\la_1=\la_2=\cdots=\la_k:=\la <0$.
We shall start to prove Theorem \ref{symmetric case}.
 \\

{\it Proof of Theorem \ref{symmetric case}}~
In the special case $N=4$, where $2p=2^*=4$, the nonlinear algebric problem \eqref{algebric system} reduces to 

\begin{equation*}
	\mu_it_i^2+\sum\limits_{j=1,j\ne i}^{k}\beta_{ij} t_j^{2} = 1,~~\hbox{for}\; i=1,2,\dots,k, 
\end{equation*}
which can be seen as
\begin{equation}\label{matrix function}
	\RR
		\begin{pmatrix}
		   t_1^2  \\ t_2^2 \\ \vdots \\ t_k^2
		\end{pmatrix}
	=
		\begin{pmatrix}
		   1  \\ 1 \\ \vdots \\ 1 
		\end{pmatrix},
\end{equation}
where $\RR=(\beta_{ij})_{k\times k} $ is the coupling matrix.

It is easy to check that if $\RR$ is invertible and the sum of each column of $\RR^{-1}$ is greater than 0, then \eqref{matrix function} has solutions
$( \pm t_1, \pm t_2, \dots, \pm t_k )$.

We can construct nontrivial solutions of \eqref{problem} as 
		\begin{equation}\label{plusminus}
			( \pm t_1 \omega, \pm t_2 \omega, \dots, \pm t_k \omega ),
		\end{equation}
where $\omega$ is the positive least energy solution of \eqref{Brezis-Nirenberg problem 1}.
Remark that these solutions are signed solutions and $(t_1 \omega, t_2 \omega, \dots, t_k \omega )$ is a positive solution of \eqref{problem}. Moreover, it's proved in \cite[Theorem 1.1]{Guo2018} that $(t_1 \omega, t_2 \omega, \dots, t_k \omega )$ is a positive least energy solution.

In the general case $N\ge 5$, Lemma \ref{existence algebric} admits a positive solution of the nonlinear algebric problem \eqref{algebric system}, we can also construct a positive solution of \eqref{problem} as
	\begin{equation*}
		(t_1\omega,t_2\omega,\dots,t_k\omega),
	\end{equation*}
where $\omega$ is the positive least energy solution of \eqref{Brezis-Nirenberg problem 1} and $(t_1,t_2,\dots,t_k)$ solves \eqref{algebric system}.

Moreover, we recall that in \cite{ambrosetti1986,capozzi1985}, the authors obtained a nontrivial solution for the Br{\'{e}}zis-Nirenberg problem \eqref{Brezis-Nirenberg problem 1} if $N\ge 4$ and $ \la<0$.
Note that \eqref{Brezis-Nirenberg problem 1} has a positive least energy solution if $-\la_1(\om)<\la<0$.
While any nontrivial solution of \eqref{Brezis-Nirenberg problem 1} is sign-changing if $\la\le -\la_1(\om)$(See \cite{Schechter2010}.)
Thus, we observe that if $\la\le -\la_1(\om)$, we can also construct a sign-changing solution of \eqref{problem} as
\begin{equation*}
		(t_1\bar{u},t_2\bar{u},\dots,t_k\bar{u}),
	\end{equation*}
where $\bar{u}$ is a nontrivial solution of \eqref{Brezis-Nirenberg problem 1} and $(t_1,t_2,\dots,t_k)$ solves \eqref{algebric system}.
	  \hfill$\Box$

\vskip 0.2in

\section{Proof of Theorem \ref{main theorem_limit}}\label{section 3}

Recall that	
	\begin{equation*}
		A:=\inf\limits_{\vu\in \NR}I(\vu)
		=\inf\limits_{\vu\in \NR}\frac{1}{N}\int\limits_{\R^N}\sum_{i=1}^{k}|\nabla u_i|^2,
	\end{equation*}
where
	\begin{equation*}
			\NR=\{\vu \in D	~:~ u_i\not\equiv 0,I'(\vu)(0,\dots,0,u_i,0,\dots,0)=0,~~i=1,2,\dots,k \}.
	\end{equation*}
Define 
	\begin{equation}\label{A'}
		A':=\inf\limits_{\vu\in \NR'}I(\vu),
	\end{equation}
where
	\begin{equation}\label{set N'}
		\NR'=\{\vu \in D \backslash \{ \vec{\textbf{0} }\}	~:~ I'(\vu)~\vu=0\}.
	\end{equation}
Note that $\NR \subset \NR' $, one has that $A'\le A$. By Sobolev inequality \eqref{sobolev embedding}, we have $A, A'>0. $

Define $B(0,R):= \{x\in \RN ~:~ |x|<R \}  $ and $H(0,R):= (H^1_0(B(0,R)))^k $.
Consider
	\begin{equation}\label{problem ball}
		\left\{
		\begin{aligned}
			-\Delta u_i&=\mu_i |u_i|^{2p-2}u_i+\sum_{j=1,j\ne i}^{k} \beta_{ij} |u_{i}|^{p-2}u_i|u_{j}|^{p}, \quad x\in B(0,R), \\
			u_i&\in H^1_0(B(0,R)), \quad i=1,2,\cdots, k,
		\end{aligned}
		\right.
	\end{equation}
and define 
	\begin{equation}\label{A'(R)}
		A'(R):=\inf\limits_{\vu\in \NR'(R)}I(\vu),
	\end{equation}
where
	\begin{equation}\label{set N'(R)}
		\begin{aligned}
		    \NR'(R)=\bigg \{ & \vu \in H(0,R) \backslash \{ \vec{\textbf{0} }\}	~:~
			\int\limits_{B(0,R)}\sum_{i=1}^{k}|\nabla u_i|^2~\mathrm{d}x\\
			  &-\int\limits_{B(0,R)}(\sum_{i=1}^{k}\mu_i|u_i|^{2p}+\sum_{i,j=1,i< j}^{k} 2\beta_{ij} |u_{i}|^{p}|u_{j}|^{p} )~\mathrm{d}x=0
			 \bigg\}.
		\end{aligned}	
	\end{equation}

\vskip 0.2in

We need the following lemma from \cite{Chen2015}, which still holds in our case.
\begin{lemma}\label{A'(R) are equivalent.}
	\quad $ A'(R)\equiv A' $ for all $R>0$.
\end{lemma}
\begin{proof}
	 Take any $R_1>R_2$. By $\NR'(R_2)\subset\NR'(R_1)$, we have $A'(R_1)\le A'(R_2)$.
	 On the other hand, for any $(u_1,u_2,\ldots,u_k)\in \NR'(R_1)$,
	 we define
	 \begin{equation*}
	 	u'_i(x):=(\frac{R_1}{R_2})^{\frac{N-2}{2}}u_i(\frac{R_1}{R_2}x),
	 \end{equation*}
	 then it is standard to see that $(u'_1,u'_2,\ldots,u'_k)\in \NR'(R_2)$, and we have
	 \begin{equation*}
	 	A'(R_2)\le I(u'_1,u'_2,\ldots,u'_k)=I(u_1,u_2,\ldots,u_k),\quad \forall~(u_1,u_2,\ldots,u_k)\in \NR'(R_1).
	 \end{equation*}
	 That is, $A'(R_2)\le A'(R_1)$ and so $A'(R_1)=A'(R_2)$.

	 Clearly $A'\le A'(R)$. Let $(u_{1n},u_{2n},\ldots,u_{kn})$ be a minimizing sequence of $A'$.
	 We may assume that $u_{1n},u_{2n},\ldots,u_{kn}\in H^1_0(B(0,R_n))$ for some $R_{n}>0$.Then $u_{1n},u_{2n},\ldots,u_{kn}\in \NR'(R_n)$, and
	 \begin{equation*}
	 	A'=\lim_{n\to\iy}I(u_{1n},u_{2n},\ldots,u_{kn})\ge \lim_{n\to\iy}A'(R_n)\equiv A'(R).
	 \end{equation*}
	 Therefore, $A'(R)\equiv A'$ for all $R>0$. 
\end{proof}

\vskip 0.2in

Let $0\le \varepsilon<p-1$. Consider
	\begin{equation}\label{problem limit ball}
		\left\{
		\begin{aligned}
			-\Delta u_i&=\mu_i |u_i|^{2p-2-2\varepsilon}u_i+\sum_{j=1,j\ne i}^{k} \beta_{ij} |u_{i}|^{p-2-\varepsilon}u_i|u_{j}|^{p-\varepsilon}, \quad x\in B(0,1), \\
			u_i&\in H^1_0(B(0,1)), \quad i=1,2,\cdots, k,
		\end{aligned}
		\right.
	\end{equation}
and define
	\begin{equation}\label{A epsilon}
		A_{\varepsilon}:=\inf\limits_{\vu\in \NR'_{\varepsilon} }I_{\varepsilon} (\vu),
	\end{equation}
where
	\begin{equation}\label{I epsilon }
		\begin{aligned}
		    I_{\varepsilon}(\vu):=&
			\frac{1}{2}\int\limits_{B(0,1) }\sum_{i=1}^{k}|\nabla u_i|^2~\mathrm{d}x\\
			-&\frac{1}{2p-2\varepsilon} \int\limits_{B(0,1) }(\sum_{i=1}^{k}\mu_i|u_i|^{2p-2\varepsilon}+\sum_{i,j=1,i< j}^{k} 2\beta_{ij} |u_{i}|^{p-\varepsilon}|u_{j}|^{p-\varepsilon} )~\mathrm{d}x,
		\end{aligned}
	\end{equation}
	\begin{equation}\label{set N epsilon}
		\NR'_{\varepsilon}=\{\vu \in H(0,1) \backslash \{ \vec{\textbf{0} }\}	~:~  H_{\varepsilon}(\vu):= I'_{\varepsilon}(\vu)~\vu=0\}.
	\end{equation}

\vskip 0.2in

We shall introduce \textbf{the idea of induction}. We notice that the proof of Lemma \ref{main lemma 1} depends on the existence result Theorem \ref{theorem=radilly in limit ball} for the $(k-1)$-coupled system case.
\vskip0.1in
From now on, we assume that Lemma \ref{main lemma 1} and Theorem \ref{theorem=radilly in limit ball} hold true for $(k-1)$-coupled system. 
We shall give out the proof of Lemma \ref{main lemma 1} and Theorem \ref{theorem=radilly in limit ball} for $ k$-coupled system in the sequel. 
Then by idea of induction, these results hold true for arbitrary $k$-coupled system.

Denote $\bu_i\in H(0,1)$ as an vector, where the number of nonzero components of $\bu_i $ is $i$ for $i=1,2,\dots,k.$ 
Denote 
\begin{equation*}
	\begin{aligned}
		C_\varepsilon(i):&=\min\{\inf\limits_{\bu_i\in\NR'_{\varepsilon} }I_{\varepsilon}(\bu_i)\}\\
		&=\min\bigg\{\inf\limits_{(u_1,\dots,u_i,0,\dots,0) \in\NR'_{\varepsilon} }I_{\varepsilon}(u_1,\dots,u_i,0,\dots,0),\dots,\\
		&\inf\limits_{(0,\dots,0,u_1,\dots,u_i) \in\NR'_{\varepsilon} }I_{\varepsilon}(0,\dots,0,u_1,\dots,u_i)    \bigg\} .
	\end{aligned}
\end{equation*}
For the case $k=2$, $C_\varepsilon(1)=\min\{\inf\limits_{(u,0)\in\NR'_{\varepsilon} }I_{\varepsilon}(u,0),\inf\limits_{(0,v)\in\NR'_{\varepsilon} }I_{\varepsilon}(0,v) \}, C_\varepsilon(2)=A_\varepsilon$.

\begin{lemma}\label{main lemma 1}
	For any $0<\varepsilon<p-1 $, there holds
	\begin{equation}
		A_{\varepsilon}<\min\big\{C_\varepsilon(1),C_\varepsilon(2),\dots,C_\varepsilon(k-1)\big\}.
	\end{equation}
\end{lemma}

\vskip 0.1in

\begin{proof}

Fix any $0<\varepsilon<p-1$. By the idea of induction, we only need to prove $$A_{\varepsilon}:=C_\varepsilon(k)<C_\varepsilon(k-1).$$
Recall that $2<2p-2\varepsilon<2^*$, we may let $(u_1,u_2,\dots,u_{k-1})$ be a least energy solution of the $(k-1)$-coupled system \eqref{problem limit ball} with energy
 $$c_{k-1}:=\inf\limits_{(u_1,u_2,\dots,u_{k-1},0)\in\NR'_{\varepsilon} }I_{\varepsilon}(u_1,u_2,\dots,u_{k-1},0).$$
We note that for any $s\in\R $, there exists a unique $t(s)>0$ such that
$$ (t(s)u_1, t(s)u_2,\dots,t(s)u_{k-1},t(s)s\phi ) \in\NR'_\varepsilon.$$
where $\phi\in H^1_0(\om)\backslash \{0\} $. In fact,
\begin{equation*}
	\begin{aligned}
	    t(s)^{2p-2\varepsilon-2} & \bigg( \int\limits_{B(0,1) }(\sum_{i=1}^{k-1}\mu_i|u_i|^{2p-2\varepsilon}+\sum_{i,j=1,i< j}^{k-1} 2\beta_{ij} |u_{i}|^{p-\varepsilon}|u_{j}|^{p-\varepsilon}\\
	    + & \mu_k |s\phi|^{2p-2\varepsilon} +\sum_{i=1}^{k-1} 2\beta_{ik}|u_{i}|^{p-\varepsilon}|s\phi|^{p-\varepsilon}) \bigg) \\
	    = & \int\limits_{B(0,1)}(\sum_{i=1}^{k-1}|\nabla u_i|^2+s^2|\nabla \phi|^2),
	\end{aligned}
\end{equation*}
Note that $t(0)=1$. Recall that $0<p-\varepsilon<2 $, by direct computations we have
\begin{equation*}
	\lim_{s\to0}\frac{t'(s)}{|s|^{p-2-\varepsilon}s}
	=-\frac{p-\varepsilon}{(p-\varepsilon-1)p'c_{k-1}}\int\limits_{B(0,1)}\sum_{i=1}^{k-1} \beta_{ik} |u_{i}|^{p-\varepsilon}|\phi|^{p-\varepsilon},
\end{equation*}
where $p'=\dfrac{2p-2\varepsilon}{p-1-\varepsilon}$. That is,
\begin{equation*}
	t'(s)=-\frac{p-\varepsilon}{(p-\varepsilon-1)p'c_{k-1}}
	\bigg(\int\limits_{B(0,1)}\sum_{i=1}^{k-1} \beta_{ik} |u_{i}|^{p-\varepsilon}|\phi|^{p-\varepsilon}\bigg)~
	|s|^{p-\varepsilon-2}s~(1+o(1)), 
\end{equation*}
$ \hbox{as}\;~s\to 0$ and so
\begin{equation*}
	t(s)=1-\frac{1}{(p-\varepsilon-1)p'c_{k-1}}
	\bigg(\int\limits_{B(0,1)}\sum_{i=1}^{k-1} \beta_{ik} |u_{i}|^{p-\varepsilon}|\phi|^{p-\varepsilon}\bigg)~
	|s|^{p-\varepsilon}~(1+o(1)),\quad \hbox{as}\;~s\to 0.
\end{equation*}
This implies that
\begin{equation*}
	\begin{aligned}
	    t(s)^{2p-2\varepsilon}
		& = 1-\frac{2p-2\varepsilon}{(p-\varepsilon-1)p'c_{k-1}}
		\bigg(\int\limits_{B(0,1)}\sum_{i=1}^{k-1} \beta_{ik} |u_{i}|^{p-\varepsilon}|\phi|^{p-\varepsilon}\bigg)~
		|s|^{p-\varepsilon}~(1+o(1)),\\
		& = 1-\frac{1}{c_{k-1}}
		\bigg(\int\limits_{B(0,1)}\sum_{i=1}^{k-1} \beta_{ik} |u_{i}|^{p-\varepsilon}|\phi|^{p-\varepsilon}\bigg)~
		|s|^{p-\varepsilon}~(1+o(1)) \quad \hbox{as}\;~s\to 0,
	\end{aligned}
\end{equation*}
then we see that
\begin{equation*}
	\begin{aligned}
	    A_\varepsilon
	    & \le I_\varepsilon(t(s)u_1,\dots,t(s)u_{k-1},t(s)s\phi )\\
	    & = \frac{1}{p'}t(s)^{2p-2\varepsilon} \bigg(p'c_{k-1}
	    + \int\limits_{B(0,1)} (\mu_k |s\phi|^{2p-2\varepsilon} +\sum_{i=1}^{k-1} 2\beta_{ik}|u_{i}|^{p-\varepsilon}|s\phi|^{p-\varepsilon}\bigg)\\
	    & = c_{k-1}-(\frac{1}{2}-\frac{1}{p'})|s|^{p-\varepsilon}\int\limits_{B(0,1)} \sum_{i=1}^{k-1} 2\beta_{ik}|u_{i}|^{p-\varepsilon}|\phi|^{p-\varepsilon}+o(|s|^{p-\varepsilon})\\
	    & <c_{k-1}\quad \hbox{as}~|s|>0 ~\hbox{small enough}.
	\end{aligned}
\end{equation*}
Then by similar arguments, we have
\begin{equation*}
	A_{\varepsilon}<C_\varepsilon(k-1).
\end{equation*}
By the idea of induction, we have $C_\varepsilon(k-1)<C_\varepsilon(k-2)<\dots<C_\varepsilon(1)$. 
This completes the proof.
\end{proof}


Denote
\begin{equation*}
	\begin{aligned}
		C(i):&=\min\{\inf\limits_{\bu_i\in\NR'}I(\bu_i)\}\\
		&=\min\bigg\{\inf\limits_{(u_1,\dots,u_i,0,\dots,0) \in\NR'}I(u_1,\dots,u_i,0,\dots,0),\dots,\\
		&\inf\limits_{(0,\dots,0,u_1,\dots,u_i) \in\NR'}I(0,\dots,0,u_1,\dots,u_i)    \bigg\} .
	\end{aligned}
\end{equation*}
Similarly as Lemma \ref{main lemma 1}, we have
	\begin{equation}\label{A'<k-1}
		\begin{aligned}
	    	A'&<\min \{C(1),C(2),\dots,C(k-1)\}\\
	    	& <C(1)\\
	    	& =\min\{I(\omega_{\mu_1},0,\dots,0),\dots,I(0,\dots,0,\omega_{\mu_k}) \}\\
	    	& =\min\{\frac{1}{N}\mu_1^{-\frac{N-2}{2}}S^{N/2},\dots,\frac{1}{N}\mu_k^{-\frac{N-2}{2}}S^{N/2}  \}.
		\end{aligned}	
	\end{equation}

\vskip 0.2in


\begin{theorem}\label{theorem=radilly in limit ball}
	For any $0<\varepsilon<p-1$,\eqref{problem limit ball}has a classical least energy solution $\vu_{\varepsilon}=(u^\varepsilon_1, u^\varepsilon_2, \dots, u^\varepsilon_k)$, and $u^\varepsilon_1, u^\varepsilon_2, \dots, u^\varepsilon_k\in C^2(B(0,1))$ are all positive radially symmetric decreasing.
\end{theorem}

\begin{proof}

	Fix any $0<\varepsilon<p-1 ,$ it is easy to see that $ A_{\varepsilon}>0$.
	For $ (u_1, u_2, \dots,u_k ) \in \NR'_{\varepsilon} $ with $ u_1, u_2, \dots,u_k \ge 0  $, we donote by $ (u_1^*, u_2^*, \dots,u_k^* ) $ as its Schwarz symmetrization. Then by the properties of Schwarz symmetrization and $\beta_{ij}>0$, we have

	\begin{equation}
		\begin{aligned}
		    \int\limits_{B(0,1) }\sum_{i=1}^{k}|\nabla u_i^*|^2~\mathrm{d}x
		    \le & \int\limits_{B(0,1) }\sum_{i=1}^{k}|\nabla u_i|^2~\mathrm{d}x\\
			=& \int\limits_{B(0,1) }(\sum_{i=1}^{k}\mu_i|u_i|^{2p-2\varepsilon}+\sum_{i,j=1,i< j}^{k} 2\beta_{ij} |u_{i}|^{p-\varepsilon}|u_{j}|^{p-\varepsilon} )~\mathrm{d}x\\
			\le & \int\limits_{B(0,1) }(\sum_{i=1}^{k}\mu_i|u_i^*|^{2p-2\varepsilon}+\sum_{i,j=1,i< j}^{k} 2\beta_{ij} |u_{i}^*|^{p-\varepsilon }|u_{j}^*|^{p-\varepsilon} )~\mathrm{d}x.
		\end{aligned}
	\end{equation}
Recall \eqref{set N epsilon} , there exists $t^*>0 $ such that $ (t^*u_1^*, t^*u_2^*, \dots,t^*u_k^* )\in \NR'_{\varepsilon} $ with
	\begin{equation*}
		(t^*)^{2p-2\varepsilon-2}=\dfrac{\int_{B(0,1) }\sum_{i=1}^{k}|\nabla u_i^*|^2~\mathrm{d}x}{\int_{B(0,1) }(\sum_{i=1}^{k}\mu_i|u_i^*|^{2p-2\varepsilon}+\sum_{i,j=1,i< j}^{k} 2\beta_{ij} |u_{i}^*|^{p-\varepsilon}|u_{j}^*|^{p-\varepsilon} )~\mathrm{d}x }\le 1,
	\end{equation*}
	and then
	\begin{equation}\label{Iepisilon=radial }
		\begin{aligned}
			I_{\varepsilon}(t^*u_1^*, t^*u_2^*, \dots,t^*u_k^* )
			& = (\frac{1}{2}-\frac{1}{2p-2\varepsilon} )(t^*)^2 \int\limits_{B(0,1) }\sum_{i=1}^{k}|\nabla u_i^*|^2~\mathrm{d}x\\
		    & \le (\frac{1}{2}-\frac{1}{2p-2\varepsilon} ) \int\limits_{B(0,1) }\sum_{i=1}^{k}|\nabla u_i|^2~\mathrm{d}x\\
		    & = I_{\varepsilon}(u_1, u_2, \dots, u_k).
		\end{aligned}
	\end{equation} 
	Therefore, we may take a minimizing sequence $\vu_n= (u_{1n}, u_{2n}, \dots, u_{kn})\in \NR'_{\varepsilon} $ such that $ (u_{1n}, u_{2n}, \dots, u_{kn}) =(u_{1n}^*, u_{2n}^*, \dots, u_{kn}^*) $ and $I_{\varepsilon}(\vu_n)\to A_{\varepsilon}  $. We see from \eqref{Iepisilon=radial }
	 $ u_{1n}, u_{2n}, \dots, u_{kn} $ are uniformly bounded in $H_{0}^1(B(0,1)) $. Passing to a subsequence, we may assume that $ u_{in}\rightharpoonup u_i^{\varepsilon}  $  weakly in  $H_{0}^1(B(0,1)) $. 
	By the compactness of the embedding $H_{0}^1(B(0,1))\hookrightarrow L^{2p-2\varepsilon}(B(0,1))  $  we have 
\begin{equation*}
		\begin{aligned}
			\int\limits_{B(0,1) } & (\sum_{i=1}^{k}\mu_i|u_i^{\varepsilon}|^{2p-2\varepsilon}+\sum_{i,j=1,i< j}^{k} 2\beta_{ij} |u_i^{\varepsilon}|^{p-\varepsilon}|u_j^{\varepsilon}|^{p-\varepsilon} )~\mathrm{d}x\\
			& =\lim_{n\to\iy}\int\limits_{B(0,1) }  (\sum_{i=1}^{k}\mu_i|u_{in} |^{2p-2\varepsilon}+\sum_{i,j=1,i< j}^{k} 2\beta_{ij} |u_{in}|^{p-\varepsilon}|u_{jn}|^{p-\varepsilon} )~\mathrm{d}x\\
			& = \dfrac{2p-2\varepsilon}{p-1-\varepsilon}\lim_{n\to\iy} I_{\varepsilon}(\vu_n)
			= \dfrac{2p-2\varepsilon}{p-1-\varepsilon} A_{\varepsilon}>0,
		\end{aligned}
	\end{equation*}
	which implies  $(u_1^{\varepsilon}, u_2^{\varepsilon},\dots, u_k^{\varepsilon} )\ne (0,0,\dots,0) $. Moreover, $u_1^{\varepsilon}, u_2^{\varepsilon},\dots, u_k^{\varepsilon}\ge 0 $ are radially symmetric.
	Meanwhile,
		\begin{equation*}
	 		\int\limits_{B(0,1) }\sum_{i=1}^{k}|\nabla u_i^{\varepsilon}|^2~\mathrm{d}x
		    \le \lim_{n\to\iy} \int\limits_{B(0,1) }\sum_{i=1}^{k}|\nabla u_{in}|^2~\mathrm{d}x,
	 	\end{equation*} 
	then
	\begin{equation*}
		\int\limits_{B(0,1) }\sum_{i=1}^{k}|\nabla u_i^{\varepsilon}|^2~\mathrm{d}x
		\le \int\limits_{B(0,1) }(\sum_{i=1}^{k}\mu_i|u_i^{\varepsilon}|^{2p-2\varepsilon}+\sum_{i,j=1,i< j}^{k} 2\beta_{ij} |u_i^{\varepsilon}|^{p-\varepsilon}|u_j^{\varepsilon}|^{p-\varepsilon} )~\mathrm{d}x.
	\end{equation*}
	Therefore, there exists $0<t_{\varepsilon}\le 1 $	such that $t_{\varepsilon}\vu_{\varepsilon}= (t_{\varepsilon}u_1^{\varepsilon}, t_{\varepsilon}u_2^{\varepsilon},\dots,t_{\varepsilon}u_k^{\varepsilon} ) \in \NR'_{\varepsilon}  $, and then
	\begin{equation*}
		\begin{aligned}
			A_{\varepsilon} & \le I_{\varepsilon}(t_{\varepsilon}\vu_{\varepsilon} )\\
			& = (\frac{1}{2}-\frac{1}{2p-2\varepsilon})(t_{\varepsilon})^2  \int\limits_{B(0,1) }\sum_{i=1}^{k}|\nabla u_i^{\varepsilon}|^2~\mathrm{d}x\\
			& \le \lim_{n\to\iy } (\frac{1}{2}-\frac{1}{2p-2\varepsilon})\int\limits_{B(0,1) }\sum_{i=1}^{k}|\nabla u_{in}|^2~\mathrm{d}x\\
			& = \lim_{n\to\iy } I_{\varepsilon}(\vu_n)= A_{\varepsilon}.
		\end{aligned}
	\end{equation*}
	Therefore, $t_{\varepsilon}=1 $ and $\vu_{\varepsilon} \in \NR'_{\varepsilon}  $ with $ I(\vu_{\varepsilon})=A_{\varepsilon} $.
	Moreover,
		\begin{equation*}
	 		\int\limits_{B(0,1) }\sum_{i=1}^{k}|\nabla u_i^{\varepsilon}|^2~\mathrm{d}x
		    = \lim_{n\to\iy} \int\limits_{B(0,1) }\sum_{i=1}^{k}|\nabla u_{in}|^2~\mathrm{d}x,
	 	\end{equation*} 
	that is, $u_{in}\to u^{\varepsilon}_i $ strongly in $H^1_0(B(0,1)) $.
	There exists a Lagrange multiplier $\gamma\in \R $ such that
	\begin{equation*}
		I'_{\varepsilon}(\vu_{\varepsilon})-\gamma H'_{\varepsilon}(\vu_{\varepsilon})=0.
	\end{equation*}
	Since $I'_{\varepsilon}(\vu_{\varepsilon})~\vu_{\varepsilon} =H_{\varepsilon}(\vu_{\varepsilon})=0 $ and
	\begin{equation*}
		H'_{\varepsilon}(\vu_{\varepsilon})~\vu_{\varepsilon}
		=(2+2\varepsilon-2p)\int\limits_{B(0,1) }(\sum_{i=1}^{k}\mu_i|u_i^{\varepsilon}|^{2p-2\varepsilon}+\sum_{i,j=1,i< j}^{k} 2\beta_{ij} |u_i^{\varepsilon}|^{p-\varepsilon}|u_j^{\varepsilon}|^{p-\varepsilon} )~\mathrm{d}x<0,
	\end{equation*}
	we get that $\gamma=0 $ and so  $I'_{\varepsilon}(\vu_{\varepsilon})=0.$
	By Lemma \ref{main lemma 1}, we see that $u^{\varepsilon}_i\not\equiv 0 $, otherwise $\vu_{\varepsilon} $ cannot minimize $I_{\varepsilon}(\vu) $. This means that $\vu_{\varepsilon} $ is a least energy solution of \eqref{problem limit ball}. Recall that $u_1^{\varepsilon}, u_2^{\varepsilon},\dots, u_k^{\varepsilon}\ge 0 $ are radially symmetric non-increasing.
	By regularity theory and the maximum principle, we see that $u_1^{\varepsilon}, u_2^{\varepsilon},\dots, u_k^{\varepsilon}> 0 $ 
	in $B(0, 1)$ and $u_1^{\varepsilon}, u_2^{\varepsilon},\dots, u_k^{\varepsilon}\in C^2(B(0,1)) $  are radially symmetric decreasing.

\end{proof}

\vskip 0.2in

Now we start to prove Theorem \ref{main theorem_limit}.\\

{\it Proof of Theorem \ref{main theorem_limit} }
	Recall the definition of $\NR'(R)$, for any $\vu\in \NR'(1)$, there exists $t_{\varepsilon}>0 $ such that $t_{\varepsilon}\vu=(t_{\varepsilon}u_1, t_{\varepsilon}u_2,\dots, t_{\varepsilon}u_k) \in \NR'_{\varepsilon}$~ with $t_{\varepsilon}\to 1 $ as $\varepsilon\to 0$. Then we have
\begin{equation*}
	\limsup_{\varepsilon\to0} A_\varepsilon
	\le \limsup_{\varepsilon\to0} I_{\varepsilon}(t_{\varepsilon}\vu)=I(\vu),\quad \forall~\vu\in\NR'(1).
\end{equation*}
By Lemma \ref{A'(R) are equivalent.} we have 
\begin{equation}\label{Ae < A'}
	\limsup_{\varepsilon\to0} A_\varepsilon \le A'(1) =A'. 
\end{equation}
By Theorem \ref{theorem=radilly in limit ball}, let $ \vu_{\varepsilon}= (u^\varepsilon_1, u^\varepsilon_2, \dots, u^\varepsilon_k) $ be a positive least energy solution of \eqref{problem limit ball}, which is radially symmetric decreasing.

Fix $0<\varepsilon\le \frac{p-1}{2}$. Then by $I'_{\varepsilon}(\vu_\varepsilon)~\vu_\varepsilon=0 $ and Sobolev inequality, we see that
\begin{equation}\label{uiepsilon uniformly bouned}
	\dfrac{2p-2\varepsilon}{p-\varepsilon-1} A_\varepsilon
	=\int\limits_{B(0,1) }\sum_{i=1}^{k}|\nabla u_i^{\varepsilon}|^2~\mathrm{d}x \ge C_0,	\quad \forall ~0<\varepsilon\le \frac{p-1}{2},
\end{equation}
where $C_0 $ is a positive constant independent of $\varepsilon$. Then 
$u^\varepsilon_1, u^\varepsilon_2, \dots, u^\varepsilon_k $ are uniformly bouned in $H_0^1(B(0,1)) $.
Passing to a subsequence,we assume that $u^\varepsilon_i\rightharpoonup u^0_i $ weakly in $H_0^1(B(0,1)) $ for $i=1,2,\dots,k. $
Then $ \vu_0= (u^0_1,u^0_2,\dots,u^0_k) $ is a solution of 
	\begin{equation}\label{problem ball R=1}
		\left\{
		\begin{aligned}
			-\Delta u_i&=\mu_i |u_i|^{2p-2}u_i+\sum_{j=1,j\ne i}^{k} \beta_{ij} |u_{i}|^{p-2}u_i|u_{j}|^{p}, \quad x\in B(0,1), \\
			u_i&\in H^1_0(B(0,1)), \quad i=1,2,\cdots, k.
		\end{aligned}
		\right.
	\end{equation}
Assume by contradiction that $\sum\limits_{i=1}^{k} \Vert u^\varepsilon_i \Vert_{\iy} $ is uniformly bounded, then by the Dominated Convergence Theorem, we get that
\begin{equation*}
	\begin{aligned}
	    \lim_{\varepsilon\to 0}\int\limits_{B(0,1) }|u_i^{\varepsilon}|^{2p-2\varepsilon}\mathrm{d}x 
	    &= \int\limits_{B(0,1) }|u_i^{0}|^{2p}\mathrm{d}x,\\
	    \lim_{\varepsilon\to 0}\int\limits_{B(0,1) }|u_i^{\varepsilon}|^{p-\varepsilon}|u_j^{\varepsilon}|^{p-\varepsilon} \mathrm{d}x 
	    &= \int\limits_{B(0,1) }|u_i^{0}|^{p}|u_j^{0}|^{p}\mathrm{d}x.
	\end{aligned}
\end{equation*}
Combining these with $I'_{\varepsilon}(\vu_\varepsilon)=I'(\vu_0)=0 $, we have that $u^{\varepsilon}_i\to u^0_i $ strongly in $H^1_0(B(0,1)) $ for $i=1,2,\dots,k. $

Then by \eqref{uiepsilon uniformly bouned}, we see that $\vu_0\ne \vec{\textbf{0} } $. Moreover $u^0_1, u^0_2,\dots, u^0_k \ge 0$. We may assume $u^0_1\not\equiv 0 $.
By the strong maximum principle, $u^0_1>0 $ in $B(0, 1)$. Note that $2p = 2^* $. Combining these with the Pohozaev identity \eqref{pohozaev}, we have
	\begin{equation*}
		    0<\int\limits_{\partial B(0,1) }\sum_{i=1}^{k} |\nabla u_i|^2 (x\cdot \vec{ \textbf{n}})~\mathrm{d}S=0,
	\end{equation*}
a contradiction. Here, $\vec{ \textbf{n}} $ denotes the outward unit normal vector on $\partial B(0,1)$.Therefore, we deduce that $\sum\limits_{i=1}^{k} \Vert u^\varepsilon_i \Vert_{\iy} \to +\iy $ as $\varepsilon\to 0 $.

Next we will use a blow up analysis. 

Note that  $u^\varepsilon_1, u^\varepsilon_2, \dots, u^\varepsilon_k $ are radially symmetric decreasing, then $u^\varepsilon_i(0)=\max\limits_{B(0,1)}u^\varepsilon_i(x) $.
We define $K_\varepsilon:=\max\limits_i\{u^\varepsilon_i(0)\} $, then $K_\varepsilon\to +\iy $ as $\varepsilon\to 0 $.
Define
\begin{equation*}
	U^\varepsilon_i(x) = K_\varepsilon^{-1}u^\varepsilon_i( K_\varepsilon^{-\alpha_\varepsilon} x),
\end{equation*}
where $\alpha_\varepsilon=p-1-\varepsilon $.
Then 
\begin{equation}\label{max=1 }
	1=\max\limits_i \{  U^\varepsilon_i(0) \}
	=\max\limits_i\bigg \{\max\limits_{x\in B(0,K_\varepsilon^{\alpha_\varepsilon})}U^\varepsilon_i(x)  \bigg\},
\end{equation}
and $U^\varepsilon_1, U^\varepsilon_2, \dots, U^\varepsilon_k $ satisfy the following system:
\begin{equation}
	-\Delta U^{\varepsilon}_i
	=\mu_i (U^\varepsilon_i)^{2p-2\varepsilon-1}+\sum_{j=1,j\ne i}^{k} \beta_{ij} (U^\varepsilon_i)^{p-1-\varepsilon}(U^\varepsilon_j)^{p-\varepsilon}, \quad x\in B(0,K_\varepsilon^{ \alpha_\varepsilon} ). 
\end{equation}
Note that $0<\varepsilon\le \frac{p-1}{2}$, then
\begin{equation*}
	\int\limits_{\RN} |\nabla U^{\varepsilon}_i|^2~\mathrm{d}x
	=K_\varepsilon^{-(N-2)\varepsilon}~\int\limits_{\RN} |\nabla u^{\varepsilon}_i|^2~\mathrm{d}x
	\le \int\limits_{\RN} |\nabla u^{\varepsilon}_i|^2~\mathrm{d}x,
\end{equation*}
then $\{(U^\varepsilon_1, U^\varepsilon_2, \dots, U^\varepsilon_k ) \}_{n\ge1} $ is bounded in $ D$.
By elliptic estimates, for a subsequence we have $(U^\varepsilon_1, U^\varepsilon_2, \dots, U^\varepsilon_k )\to (U_1, U_2, \dots, U_k ) \in D$ 
uniformly in every compact subset of $ \RN$ as $\varepsilon\to0 $, and $ (U_1, U_2, \dots, U_k )$ satisfies \eqref{RN system}.Then we see that $I'(U_1, U_2, \dots, U_k )=0 $. Moreover, $U_1, U_2, \dots, U_k \ge 0 $ are radially symmetric decreasing.

By \eqref{max=1 } we have $(U_1, U_2, \dots, U_k )\ne (0,\dots, 0) $, and so $(U_1, U_2, \dots, U_k )\in \NR' $.
Then we deduce from \eqref{Ae < A'} that
\begin{equation*}
	\begin{aligned}
		A' & \le I(U_1, U_2,\dots, U_k)
		=(\frac{1}{2}-\frac{1}{2p} ) \int\limits_{\RN}\sum_{i=1}^{k} |\nabla U_i|^2~\mathrm{d}x\\
		& \le \liminf\limits_{\varepsilon\to0}~ (\frac{1}{2}-\frac{1}{2p-2\varepsilon } ) \int\limits_{B(0,K_\varepsilon^{\alpha_\varepsilon}) }\sum_{i=1}^{k} |\nabla U_i^\varepsilon|^2~\mathrm{d}x\\
		& \le \liminf\limits_{\varepsilon\to0}~ (\frac{1}{2}-\frac{1}{2p-2\varepsilon } ) \int\limits_{B(0,1) }\sum_{i=1}^{k} |\nabla u_i^\varepsilon|^2~\mathrm{d}x\\
		& =\liminf\limits_{\varepsilon\to0} A_\varepsilon \le A'.
	\end{aligned}
\end{equation*}
This implies that $I(U_1, U_2, \dots, U_k)=A'$. By \eqref{A'<k-1} we have that $ U_i\not\equiv 0 $ for all $i=1,2,\dots,k $.
By strong maximum principle, $U_1, U_2, \dots, U_k > 0$ are radially symmetric decreasing.
Notice that $(U_1, U_2, \dots, U_k)\in\NR $ and so $I(U_1, U_2, \dots, U_k)\ge A\ge A' $, that is,
\begin{equation}\label{A=A'}
	I(U_1, U_2, \dots, U_k)=A'=A,
\end{equation}
and $(U_1, U_2, \dots, U_k) $ is a positive least energy solution of \eqref{RN system}, which is radially symmetric decreasing.

\vskip 0.2in

Now we shall prove that the positive least energy solution of \eqref{RN system} must be the least energy synchronized type solution of the form
$(t_1\omega,t_2\omega,\dots,t_k\omega)$. Recall the minimizing problem
	\begin{equation*}
		A':=\inf\limits_{\vu\in \NR'}I(\vu),
	\end{equation*}
where
	\begin{equation*}
		\NR'=\{\vu \in D \backslash \{ \vec{\textbf{0} }\}	~:~ I'(\vu)~\vu=0\},
	\end{equation*}
and the minimizing problem introducd in the proof of Lemma \ref{existence algebric}:
	\begin{equation*}
		d_k:=\inf\limits_{\vt\in\PR_k }\GR_k(\vt),
	\end{equation*}
	where
	\begin{equation*}
		\GR_k(\vt)=\sum\limits_{i=1}^k(\frac{t_i^2}{2}-\frac{\mu_i|t_i|^{2p}}{2p})- \frac{1}{2p}\big(\sum\limits_{i=1,j>i}^{k}2\beta_{ij}|t_i|^{p}|t_j|^{p} \big),
	\end{equation*}
	\begin{equation*}
		\PR_k= \{\vt\in\R^k \backslash\{\vec{\textbf{0} } \} ~|~P_k(\vt):=\sum\limits_{i=1}^k(t_i^2-\mu_i|t_i|^{2p})-  \sum\limits_{i=1,j>i}^{k}2\beta_{ij}|t_i|^{p}|t_j|^{p}=0   \}.
	\end{equation*}
By standard argument (cf. \cite{peng2016elliptic}), we see that
\begin{equation*}
	d_k=\inf_{\vt\in\R^k\backslash\{\vec{\textbf{0}}\}}
	\dfrac{ \big(\sum_{i=1}^k t_i^2\big)^{\frac{N}{2}}}{N\big(\sum\limits_{i=1}^k\mu_i|t_i|^{2p}-  \sum\limits_{i=1,j>i}^{k}2\beta_{ij}|t_i|^{p}|t_j|^{p} \big)^{\frac{N-2}{2} } },
\end{equation*}
and by \eqref{A=A'} we have
\begin{equation}\label{minimizing A'-1}
	A=A'=\inf_{\vu\in D\backslash\{\vec{\textbf{0}}\}}
	\dfrac{ \big(\sum_{i=1}^{k}\Vert \nabla u_i\Vert_2^2~ \big)^{\frac{N}{2}}}
	{N\big(\sum_{i=1}^{k}\mu_i\Vert u_i\Vert^{2p}_{2p}+\sum_{i,j=1,i< j}^{k} 2\beta_{ij} \Vert u_{i}u_{j}\Vert^{p}_{p} \big)^{\frac{N-2}{2} } },
\end{equation}
then by Sobolev inequality \eqref{sobolev embedding}, we have
\begin{equation*}
	A'\ge \inf_{\vu\in D\backslash\{\vec{\textbf{0}}\}}
	\dfrac{ \big(\sum_{i=1}^{k}\Vert u_i\Vert_{2p}^2~ \big)^{\frac{N}{2}}}
	{N\big(\sum_{i=1}^{k}\mu_i\Vert u_i\Vert^{2p}_{2p}+\sum_{i,j=1,i< j}^{k} 2\beta_{ij} \Vert u_{i}\Vert^p_{2p} \Vert u_{j}\Vert^{p}_{2p} \big)^{\frac{N-2}{2} } }S^{\frac{N}{2}},
\end{equation*}
where $\Vert\cdot\Vert_{q}$ denotes the norm on $L^{q}(\RN)$ and $2p=2^*=\frac{2N}{N-2} $.
Recall Lemma \ref{existence algebric}, then we see that 
\begin{equation*}
	A'=d_kS^\frac{N}{2}
\end{equation*}
is attained by
\begin{equation*}
	(t_1U_{\varepsilon,y},t_2U_{\varepsilon,y},\dots,t_kU_{\varepsilon,y}),
\end{equation*}
where $(t_1, t_2, \dots, t_k) $  is the positive solution of \eqref{algebric system} that attains $d_k$, and $ U_{\varepsilon,y}$ is the Aubin-Talenti instanton given by \eqref{Aubin-Talenti}.

In fact, we deduce that
\begin{equation}\label{minimizing A'-2}
	A'= \inf_{\vu\in D\backslash\{\vec{\textbf{0}}\}}
	\dfrac{ \big(\sum_{i=1}^{k}\Vert u_i\Vert_{2p}^2~ \big)^{\frac{N}{2}}}
	{N\big(\sum_{i=1}^{k}\mu_i\Vert u_i\Vert^{2p}_{2p}+\sum_{i,j=1,i< j}^{k} 2\beta_{ij} \Vert u_{i}\Vert^p_{2p} \Vert u_{j}\Vert^{p}_{2p} \big)^{\frac{N-2}{2} } }S^{\frac{N}{2}}.
\end{equation}
Suppose that $A'$ is attained  by some nonzero $\vec{\textbf{v}}=(v_1,v_2,\dots,v_k)$.
Then by H\"older inequality and Sobolev inequality, we see from \eqref{minimizing A'-1} and \eqref{minimizing A'-2} that 
 \begin{equation*}
 	 \Vert\nabla v_i\Vert^2_2=S\Vert v_i\Vert^2_{2p} 
 \end{equation*}
 for all $i=1,2,\dots,k$, which implies $v_i=c_iU_{\varepsilon,y} $ for some $c_i\ne0, \varepsilon>0$ and $y\in\RN$, or $v_i=0$.
 Redenote all $v_i=c_iU_{\varepsilon,y} $ for all $i$ where $c_{i}=0 $ or $c_{i}\ne0$.

By \eqref{minimizing A'-2}, we have that $ (\Vert v_1\Vert_{2p}, \Vert v_2\Vert_{2p}, \dots, \Vert v_k\Vert_{2p} )$ attains $d_k$. 
 It is easy to see that
 $(c_1, c_2,\dots, c_k) $ attains $d_k$.
 Recall Lemma \ref{existence algebric} that $d_k $ is attained by a positive solution if $\beta_{ij}>0$. 
 Then $c_i>0$ for all $i=1,2,\dots,k$.

That is, $A'$ is attained by $\vu$ if and only if
\begin{equation*}
	\vu=(\tilde{t}_1U_{\varepsilon,y},\tilde{t}_2U_{\varepsilon,y},\dots,\tilde{t}_kU_{\varepsilon,y}),
\end{equation*}
where $ (\tilde{t}_1, \tilde{t}_2, \dots, \tilde{t}_k)$ is defined in \eqref{least algebric solution}.
We see that, if $\vec{\textbf{v}} \in \NR$ attains $A$, then $ \vec{\textbf{v}}\in\NR' $ and by \eqref{A=A'} we have
	$$   I(\vec{\textbf{v}})=A=A',  $$ 
which means $\vec{\textbf{v}}$ attains $A'$.
Since $A'$ is attained only by $\vu=(\tilde{t}_1U_{\varepsilon,y},\tilde{t}_2U_{\varepsilon,y},\dots,\tilde{t}_kU_{\varepsilon,y})$, which solves \eqref{RN system}, then $\vu\in \NR$, by
$$   I(\vu)=A=A',  $$ 
we see that $\vu$ also attains $A$.
Hence, the positive least energy solution of \eqref{RN system} must be the least energy synchronized type $(\tilde{t}_1U_{\varepsilon,y},\tilde{t}_2U_{\varepsilon,y},\dots,\tilde{t}_kU_{\varepsilon,y})$, where $(\tilde{t}_1, \tilde{t}_2,\dots,\tilde{t}_k)$ is defined in \eqref{least algebric solution}.

\vskip 0.2in

Finally, we show the existence of $(t_1(\beta),t_2(\beta),\dots,t_k(\beta))$ for $\beta_{ij}=\beta>0 $ small.
Denote functions
	\begin{equation}
		\begin{aligned}
	   	 f_i(t_1,t_2,\dots,t_k,\beta)
	   	 & :=\mu_it_i^{2p-2}\\
	   	 &+\sum\limits_{j=1,j\ne i}^{k}\beta t_i^{p-2}t_j^{p} - 1, \quad t_i>0,~~\hbox{and}~~t_j\ge 0 ~~\hbox{for}~~ j\ne i,
		\end{aligned}		
	\end{equation}
and define $t_i(0)=\mu_i^{-\frac{1}{2p-2} } $, then $ f_i(t_1(0),t_2(0) ,\dots,t_k(0),0)=0$ for all $ i=1,2,\dots,k. $ Note that
\begin{equation*}
	\begin{aligned}
	    \partial_i f_i(t_1(0),t_2(0) ,\dots,t_k(0) ,0)&=(2p-2)\mu_i(t_i(0))^{ 2p-3 }>0,\\
	\partial_j f_i(t_1(0),t_2(0) ,\dots,t_k(0) ,0)&=0,\quad \hbox{if}~~ j\ne i~~\hbox{and}~~j=1,2,\dots,k,
	\end{aligned}
\end{equation*}
which implies that 
	\begin{equation*}
		\det(F_{ij})_{k\times k}:= \det ( \partial_j f_i(t_1(0),t_2(0) ,\dots,t_k(0) ,0))_{k\times k}>0.
	\end{equation*}
Therefore, by the implicit function theorem, $(t_1(\beta),t_2(\beta),\dots,t_k(\beta))$ are well defined and class $C^1 $ on $(-\beta_2,\beta_2) $ for some $\beta_2>0$, and
	\begin{equation*}
		f_i(t_1(\beta), t_2(\beta) ,\dots, t_k(\beta) ,\beta) \equiv 0,\quad i=1,2,\dots,k.  
	\end{equation*}
This implies that $(t_1(\beta), t_2(\beta) ,\dots, t_k(\beta)) $ solves \eqref{algebric system}. Moreover, we notice that
 ~$(t_1(\beta)U_{\varepsilon,y},t_2(\beta)U_{\varepsilon,y},\dots,t_k(\beta)U_{\varepsilon,y}) $ is a positive solution of \eqref{RN system} and

 \begin{equation*}
 	\lim_{\beta\to0}\sum\limits_{i=1}^k t_i^2(\beta) = \sum\limits_{i=1}^k t_i^2(0)= \sum\limits_{i=1}^k \mu_i^{-\frac{1}{p-1} },
 \end{equation*}
that is, there exists $0<\beta_1\le \beta_2 $, such that
\begin{equation*}
 	\sum\limits_{i=1}^k t_i^2(\beta) > \min\limits_i \{\mu_i^{-\frac{1}{p-1} } \} ,\quad \forall \beta\in (0,\beta_1).
\end{equation*}
Combining this with \eqref{A'<k-1}, we have
\begin{equation*}
	I(\vec{\textbf{U} } )=A'=A<I(t_1(\beta)U_{\varepsilon,y},t_2(\beta)U_{\varepsilon,y},\dots,t_k(\beta)U_{\varepsilon,y}),\quad \forall \beta\in (0,\beta_1),
\end{equation*}
that is, $ (t_1(\beta)U_{\varepsilon,y},t_2(\beta)U_{\varepsilon,y},\dots,t_k(\beta)U_{\varepsilon,y}) $ is a different positive solution of \eqref{RN system} with respect to $ \vec{\textbf{U}} $. This completes the proof.
\hfill $\Box $

\vskip 0.2in

Similarly to \cite[Proposition 2.1]{Chen2015}, the following properties still hold in our case.

\begin{proposition}
	Assume that $\mu_i>0$ and $\beta_{ij}>0 $. Let $\vec{\textbf{U}}=(U_1,U_2,\dots,U_k) $ be a positive radially symmetric least energy solution of \eqref{RN system} obtained in Theorem \ref{main theorem_limit}. Then there exists $C>0 $ such that
	\begin{equation*}
		\sum\limits_{i=1}^k U_i(x)\le C(1+|x| )^{2-N},
		\quad \sum\limits_{i=1}^k |\nabla U_i(x)|\le C(1+|x| )^{2-N}.
	\end{equation*}
\end{proposition}

\begin{proof}
	Define the Kelvin transformation:
	\begin{equation*}
		U^*_i(x):=|x|^{2-N}U_i(x^*),\quad x^*=\frac{x}{|x|^2},
	\end{equation*}
then $U^{*}_{1},\ldots,U^{*}_{k}\in D^{1,2}(\RN)$  and $ (U^{*}_{1},\ldots,U^{*}_{k}) $ satisfies the same system \eqref{RN system}.
By a standard Br{\'e}zis-Kato argument \cite{Brezis1979}, we see that $U^{*}_{1},\ldots,U^{*}_{k}\in L^\iy(\RN)$.
Therefore, there exists $C>0$ such that
\begin{equation*}
	U^*_i(x^*)=|x|^{N-2}U_i(x)\le C,
\end{equation*}
then
\begin{equation}\label{radial U_i bounded}
	\sum\limits_{i=1}^k U_i(x)\le C|x|^{2-N}.
\end{equation}
On the other hand, we note that $U_1,U_2,\dots,U_k$ are radially symmetric decreasing. We also have $U_1,U_2,\dots,U_k\in L^\iy(\RN)$, and then
\begin{equation*}
	\sum\limits_{i=1}^k U_i(x)\le C(1+|x| )^{2-N}.
\end{equation*}
Moreover, by standard elliptic regularity theory, we have that $U_1,U_2,\dots,U_k\in C^2(\RN)$. We write $U_i(|x|)=U_i(x)$ for convinience. Then
\begin{equation*}
	(r^{N-1} (U_i)_r)_r=-r^{N-1}(\mu_iU_i^{2^*-1}+\sum_{j=1,j\ne i}^{k}\beta_{ij}U_i^{\frac{2^*}{2}-1}U_j^{\frac{2^*}{2}}),
\end{equation*}
and so for any $R\ge 1$, we see from \eqref{radial U_i bounded} that
\begin{equation*}
	\begin{aligned}
	    R^{N-1}|(U_i)_r(R)|
	    & \le |(U_i)_r(1)|+\int_{1}^{R}r^{N-1}(\mu_iU_i^{2^*-1}+\sum_{j=1,j\ne i}^{k}\beta_{ij}U_i^{\frac{2^*}{2}-1}U_j^{\frac{2^*}{2}})~\mathrm{d}r\\
	    & \le C+C\int_{1}^{+\iy}r^{N-1}r^{-N-2}~\mathrm{d}r \le C.
	\end{aligned}
\end{equation*}
Therefore, 
it is easy to see that $|\nabla U_i(x)|\le C(1+|x| )^{2-N}$ for some $C>0$.
Then there exists $C>0$ such that
\begin{equation*}
	\sum\limits_{i=1}^k |\nabla U_i(x)|\le C(1+|x| )^{2-N}.
\end{equation*}
\end{proof}

\vskip 0.1in

\section{Proof of Theorem \ref{main theorem} }\label{section 4}

In this section, without loss of genrality, we assume that 
$$-\lambda_1(\om)<\la_1\le \la_2\le\dots\le \la_k <0 .$$
Recall the definition of $B $ in \eqref{least energy B}, since
\begin{equation*}
	\int\limits_{\om} (|\nabla u|^2+\la_iu^2)~\mathrm{d}x
	\ge (1+\frac{\la_i}{\la_1(\om)} ) \int\limits_{\om} |\nabla u|^2~\mathrm{d}x,\quad i=1,2,\dots,k,
\end{equation*}
it is standard to see that $B>0 $.
Recall \cite{Brezis1983} that the Br{\'{e}}zis-Nirenberg problem
\begin{equation}\label{Brezis-Nirenberg problem mu}
	-\Delta u+\lambda_i u=\mu_i |u|^{2^*-2}u, \quad u\in H^1_0(\om) 	
\end{equation}
has a positive least energy solution $u_{\mu_i} \in C^2(\Omega)\cap C(\bar{\Omega} ) $ with energy
\begin{equation}\label{Brezis-Nirenberg least energy mu}
	\begin{aligned}
		B_{\mu_i}:= & \dfrac{1}{2}\int\limits_{\om}(|\nabla u_{\mu_i}|^2+\lambda_iu_{\mu_i}^2)~\mathrm{d}x-\frac{\mu_i}{2^*}\int\limits_{\om} u_{\mu_i}^{2^*} ~\mathrm{d}x,\\
		< & \dfrac{1}{N}\mu_i^{-\frac{N-2}{2}}S^{N/2}, \quad i=1,2,\dots,k.
	\end{aligned}
\end{equation}
We need the following lemma from \cite{Chen2015}.

\begin{lemma}\label{Brezis-Lieb type lemma }
	Let $u_n\rightharpoonup u, v_n \rightharpoonup v $ in $H_0^1(\om) $
as $n\to \iy $, then passing to a subsequence, there holds
\begin{equation*}
 	\lim_{n\to\iy} \int\limits_\om (|u_n|^p|v_n|^p-|u_n-u|^p|v_n-v|^p-|u|^p|v|^p ) ~\mathrm{d}x=0.
 \end{equation*} 
\end{lemma}

\vskip 0.1in

\begin{remark}
	Similar to the proof of Theorem \ref{main theorem_limit} in Sect. \ref{section 3}, before starting the proof, we shall introduce \textbf{the idea of induction} again.
	From now on, we assume that Theorem \ref{main theorem} hold true for $(k-1)$-coupled system, that is,

	{\textbf{ the $(k-1)$-coupled system has a positive least energy solution. }}\\
	We shall give out the proof of these results for $ k$-coupled system in the sequel. Then by the idea of induction, these theorems hold true for arbitrary $k$-coupled system. 
\end{remark}

\vskip 0.1in
For convinience, here are some notations. Consider the $k$-coupled system \eqref{problem}, where the coefficients $\la_i, \mu_i, \beta_{ij}$ are fixed.
For each $\Theta\subsetneqq \{1, . . . ,k\}$ , we obtain the following $(k -|\Theta |)$-coupled system
	\begin{equation}\label{problem k}
		\left\{
		\begin{aligned}
			-\Delta u_i & +\lambda_iu_i=\mu_i u_i^{2p-1}+\sum_{j=1,j\ne i}^{k} \beta_{ij} u_{i}^{p-1}u_{j}^{p} \quad \hbox{in}\;\Omega,\\
			u_i&>0~~ \hbox{in}\; \Omega \quad \hbox{and}\quad u_i=0~~ \hbox{on}\;\partial\Omega, \quad i,j\in \{1,2,\cdots, k\}\backslash \Theta,
		\end{aligned}
		\right.
	\end{equation}
by 
replacing $\la_i, \mu_i,\beta_{ij},\beta_{ji}$ with 0 if $i \in \Theta $, where $ |\Theta| $ is the cardinality of $\Theta $. 
Note that the nontrivial solutions of \eqref{problem k} can be extended to the semi-trivial solutions of $k$-coupled system \eqref{problem}.
\vskip 0.1in
Denote $\tau=k-|\Theta|$. Indeed, we can establish $C_{k}^{\tau}$ different $\tau$-coupled systems from \eqref{problem}.
For every fixed $\Theta$, we denote $\{i_1,\dots,i_{\tau}\}=\{1,2,\dots,k\}\backslash \Theta$ and $ B_{\mu_{i_1}\dots\mu_{i_\tau}} $ as the least energy of the $\tau $-coupled system \eqref{problem k} with coefficients $(\mu_{i_1},\dots,\mu_{i_\tau})$, that is,
\begin{equation}
	B_{\mu_{i_1}\dots\mu_{i_\tau}}:=\inf\{E(\vu)~:~\vu=(u_1,\dots,u_k)~\hbox{solves}~\eqref{problem}~~\hbox{and}~u_i=0~\hbox{iff}~i\in \Theta \}.
\end{equation}
Denote
\begin{equation}
	\bar{ B}_\tau :=	\min\limits_{\forall \{l_1,\dots,l_{\tau}\}\subsetneqq\{1,2,\dots,k\} } \{B_{\mu_{l_1}\dots\mu_{l_\tau}} \}.
\end{equation}
Define
\begin{equation}\label{mountain B}
	\hb:=\inf_{h\in \Gamma}\max_{t\in[0,1]}E(h(t)),
\end{equation}
where $\Gamma=\{h\in C([0,1],H)~:~h(0)=\vec{ \textbf{0}},~E(h(1))<0 \}. $
By \eqref{energy functional E}, we see that for any $\vu\in H, \vu\ne \vec{ \textbf{0}} $,
\begin{equation}\label{maxE}
	\begin{aligned}
		\max_{t>0}E(t\vu)& =E(t_{\vu} \vu )\\
		& = \frac{1}{N} t^2_{\vu}\int\limits_{\om}\sum_{i=1}^{k}(|\nabla u_i|^2+\lambda_iu_i^2)\\
	    &=\dfrac{1}{N} t^{2^*}_{\vu} \int\limits_{\om}(\sum_{i=1}^{k}\mu_i|u_i|^{2p}+\sum_{i=1,j>i}^{k} 2\beta_{ij} |u_{i}|^{p}|u_{j}|^{p} ),
	\end{aligned}
\end{equation}
where $t_{\vu}>0$ satisfies
\begin{equation}\label{tu function}
	t_{\vu}^{2p-2}=\dfrac{\int_{\om}\sum_{i=1}^{k}(|\nabla u_i|^2+\lambda_iu_i^2) }{\int_{\om}(\sum_{i=1}^{k}\mu_i|u_i|^{2p}+\sum_{i=1,j>i}^{k} 2\beta_{ij} |u_{i}|^{p}|u_{j}|^{p} ) }.
\end{equation}
Note that $t_{\vu}~\vu\in \MR' $, where
\begin{equation}\label{M' definition}
	\begin{aligned}
		\MR' :=\bigg\{  \vu\in H\backslash \{\vec{\textbf{0} } \}~ & :~
		G(\vu)  := \int\limits_{\om}\sum_{i=1}^{k}(|\nabla u_i|^2+\lambda_iu_i^2)\\
		& -  \int\limits_{\om}(\sum_{i=1}^{k}\mu_i|u_i|^{2p}+\sum_{i=1,j>i}^{k} 2\beta_{ij} |u_{i}|^{p}|u_{j}|^{p} )  = 0	\bigg\}.
	\end{aligned}
\end{equation}
Note that $\MR\subset\MR' $, one has that $0< \hb\le B $. It is easy to check that
\begin{equation}\label{hb E}
	\hb=\inf_{\vec{\textbf{0} }\ne \vu\in H }\max_{t>0}E(t \vu)
	=\inf_{\vu\in\MR'}E(\vu).
\end{equation}
In fact, by \eqref{maxE} we have that 
\begin{equation*}
	\inf\limits_{\vec{\textbf{0} }\ne \vu\in H }\max\limits_{t>0}E(t \vu)=\inf\limits_{\vu\in\MR'}E(\vu),
\end{equation*}
and
\begin{equation*}
	\begin{aligned}
    E'(t\vu)(t\vu)
    & =t^2\int_{\om}\sum_{i=1}^{k}(|\nabla u_i|^2+\lambda_iu_i^2)  -t^{2p}\int_{\om}(\sum_{i=1}^{k}\mu_i|u_i|^{2p}+\sum_{i=1,j>i}^{k} 2\beta_{ij} |u_{i}|^{p}|u_{j}|^{p} )\\
    & \ge 0 \quad \hbox{if}~~~0\le t\le t_{\vu}~.
	\end{aligned}
\end{equation*}
On the one hand, we see that $E(t\vu)<0$ for $\vu\in H\backslash\{\vec{\textbf{0}}\} $ and $t>0$ large enough. Then by the definition of $\hb$, we have
\begin{equation*}
	\hb\le \inf\limits_{\vec{\textbf{0} }\ne \vu\in H }\max\limits_{t>0}E(t \vu).
\end{equation*}
On the other hand, $\MR'$ separates $H$ into two components
\begin{equation*}
	\begin{aligned}
		&\MR'_1=\{\vu\in H\backslash\{\vec{\textbf{0}}\}~:~E'(\vu)\vu> 0\} \cup \{\vec{\textbf{0}}\},\\
		&\MR'_2=\{\vu\in H\backslash\{\vec{\textbf{0}}\}~:~E'(\vu)\vu<0\}.
	\end{aligned}
\end{equation*}
Since $E'(t\vu)(t\vu)\ge 0$ if $0\le t\le t_{\vu}$, we have
\begin{equation*}
	\MR'_1=\{t\vu~:~ \vu\in H\backslash\{\vec{\textbf{0}}\},~0\le t< t_{\vu}\},\quad \MR'_2=\{t\vu~:~ \vu\in H\backslash\{\vec{\textbf{0}}\},~t>t_{\vu}\}.
\end{equation*}
Note that for any $\vec{\textbf{0}}\ne t\vu\in \MR'_1$,
\begin{equation*}
	\begin{aligned}
		E(t\vu) & = \frac{t^2}{2}\int_{\om}\sum_{i=1}^{k}(|\nabla u_i|^2+\lambda_iu_i^2)-\frac{t^{2p}}{2p}\int_{\om}(\sum_{i=1}^{k}\mu_i|u_i|^{2p}+\sum_{i=1,j>i}^{k} 2\beta_{ij} |u_{i}|^{p}|u_{j}|^{p} )\\
		& >\frac{1}{2p}\bigg[t^2\int_{\om}\sum_{i=1}^{k}(|\nabla u_i|^2+\lambda_iu_i^2)  -t^{2p}\int_{\om}(\sum_{i=1}^{k}\mu_i|u_i|^{2p}+\sum_{i=1,j>i}^{k} 2\beta_{ij} |u_{i}|^{p}|u_{j}|^{p} ) \bigg]\\
		& = \frac{1}{2p}E'(t\vu)(t\vu)> 0,\quad \hbox{if}~~0< t< t_{\vu},
	\end{aligned}
\end{equation*}
then every $h \in \Gamma$ has to cross $\MR'$ and $\inf\limits_{\vu\in\MR'}E(\vu)\le \hb$.
It follows that \eqref{hb E} holds.

\vskip 0.2in

Next we introduce the key lemma in the proof of Theorem \ref{main theorem} which gives out an estimate of $ \hb $.

\begin{lemma}\label{key lemma}(the comparison lemma) Let $\beta_{ij}>0 $, then
	$$ \hb<\min \{ \bar{B}_1, \bar{B}_2,\dots,\bar{B}_{k-1}, A \}. $$
\end{lemma}

\begin{proof}
We will prove the lemma in three steps.
\vskip 0.1in

\noindent\textbf{Step 1}\quad ( $\hb<A $ )

Without loss of genrality, we may assume that $0\in\Omega. $
Then there exists $\rho>0 $ such that $ B(0,2\rho):=\{x~:~|x|\le 2\rho \}\subset \Omega $.
Let $\psi\in C_0^1(B(0,2\rho)) $ be a nonnegative function with $0\le \psi\le1 $ and $\psi\equiv 1 $ for $|x| \le \rho.  $
Recall that $\vec{\textbf{U}}=(U_1,U_2,\dots,U_k) $ in Theorem \ref{main theorem_limit}. We define
\begin{equation*}
	U_i^{\varepsilon}(x):=\varepsilon^{-\frac{N-2}{2}}U_i(x/\varepsilon ),
	\quad
	u_i^{\varepsilon}:=\psi U_i^{\varepsilon},
	\quad i=1,2,\dots, k.
\end{equation*}
Then it is easy to see that
\begin{equation*}
	\int\limits_{\RN}|\nabla U_i^{\varepsilon}|^2
	=\int\limits_{\RN}|\nabla U_i|^2,	\quad
	\int\limits_{\RN}|U_i^{\varepsilon}|^{2^*}
	=\int\limits_{\RN}|U_i|^{2^*},\quad
	i=1,2,\dots,k.
\end{equation*}
It's proved in \cite{Chen2015} that
\begin{equation}\label{estimate 1}
	\int\limits_{\om}|\nabla u_i^{\varepsilon}|^2 
	\le \int\limits_{\RN}|\nabla U_i|^2+O(\varepsilon^{N-2}) ,
\end{equation}
\begin{equation}\label{estimate 2}
	\int\limits_{\om}|u_i^{\varepsilon}|^{2^*}
	\ge \int\limits_{\RN}|U_i|^{2^*}+O(\varepsilon^{N}) ,
\end{equation}
\begin{equation}\label{estimate 3}
	\int\limits_{\om}|u_i^{\varepsilon}|^{\frac{2^*}{2}}|u_j^{\varepsilon}|^{\frac{2^*}{2}}
	\ge \int\limits_{\RN}|U_i|^{\frac{2^*}{2}}|U_j|^{\frac{2^*}{2}} +O(\varepsilon^{N}) ,
\end{equation}
\begin{equation}\label{estimate 4}
	\int\limits_{\om}|u_i^{\varepsilon}|^{2}
	\ge C\varepsilon^2+O(\varepsilon^{N-2}) ,
\end{equation}
where $ C$ is a positive constant. Recall that $I(U_1,U_2,\dots,U_k)=A $, we have
\begin{equation*}
	NA=\int\limits_{\RN} \sum_{i=1}^{k}|\nabla U_i |^{2}
	=\int\limits_{\RN}(\sum_{i=1}^{k}\mu_i|U_i|^{2^*}+\sum_{i=1,j>i}^{k} 2\beta_{ij} |U_{i}|^{\frac{2^*}{2}}|U_{j}|^{\frac{2^*}{2}} ).
\end{equation*}
Combining this with \eqref{estimate 1}-\eqref{estimate 4} and recalling that $\la_i<0,~2p=2^*,~N\ge 5$, we have for any $ t>0$ that
\begin{equation*}
	\begin{aligned}
		E(tu_1^{\varepsilon},\dots,tu_k^{\varepsilon})
    & =\frac{1}{2}t^2\int\limits_{\om}\sum_{i=1}^{k}(|\nabla u_i^{\varepsilon} |^2+\lambda_i(u^{\varepsilon}_i)^2)\\
	& -\dfrac{1}{2p} t^{2p} \int\limits_{\om}(\sum_{i=1}^{k}\mu_i|u_i^{\varepsilon}|^{2p}+\sum_{i=1,j>i}^{k} 2\beta_{ij} |u_i^{\varepsilon}|^{p}|u_j^{\varepsilon}|^{p} )\\
	& \le \frac{1}{2}\bigg(	\int\limits_{\RN} \sum_{i=1}^{k}|\nabla U_i |^{2} -C\varepsilon^2+O(\varepsilon^{N-2}) \bigg)t^2\\
	& -\frac{1}{2^*}\bigg(\int\limits_{\RN}(\sum_{i=1}^{k}\mu_i|U_i|^{2^*}+\sum_{i=1,j>i}^{k} 2\beta_{ij} |U_{i}|^{\frac{2^*}{2}}|U_{j}|^{\frac{2^*}{2}} ) +O(\varepsilon^N)  \bigg) t^{2^*}\\
	& =\frac{1}{2}\bigg(NA-C\varepsilon^2+O(\varepsilon^{N-2}) \bigg)t^2
	-\frac{1}{2^*}\bigg(NA+O(\varepsilon^N)  \bigg) t^{2^*}\\
	& \le \frac{1}{N}\bigg(NA-C\varepsilon^2+O(\varepsilon^{N-2}) \bigg)\bigg(\frac{NA-C\varepsilon^2+O(\varepsilon^{N-2}) }{NA+O(\varepsilon^{N})}  \bigg)^{\frac{N-2}{2}}\\
	&<A \quad \hbox{for} ~~\varepsilon>0~~\hbox{small enough}.
	\end{aligned}
\end{equation*}
Hence, for $\varepsilon>0 $ small enough, there holds
\begin{equation}
	\hb \le \max_{t>0}E(tu_1^{\varepsilon},\dots,tu_k^{\varepsilon}) <A.
\end{equation} 

From now on, we redenote $\hb_k:=\hb $ as the minimax value defined in \eqref{mountain B} for the $k$-coupled system.

\noindent\textbf{Step 2}\quad ($ k=3 $ and $\hb_3<\min \{\bar{B}_1, \bar{B}_2\} $)  

In \cite{Chen2015}, it's proved that the $2$-coupled system
	\begin{equation}\label{problem 2}
		\left\{
		\begin{aligned}
			-\Delta u+\lambda_iu&=\mu_i u^{2p-1}+\beta_{ij} u^{p-1}v^{p}, \quad x\in\Omega,\\
			-\Delta v+\lambda_jv&=\mu_j v^{2p-1}+\beta_{ij} v^{p-1}u^{p}, \quad x\in\Omega,\\
			u,v&\ge0~~ \hbox{in}\; \Omega \quad \hbox{and}\quad u=v=0~~ \hbox{on}\;\partial\Omega,
		\end{aligned}
		\right.
	\end{equation}
has a positive least energy solution $(u_{ij}, v_{ij})$ with energy $B_{\mu_i\mu_j} $ and
\begin{equation}
	B_{\mu_i\mu_j}<\min \{ B_{\mu_i},B_{\mu_j}  \},
\end{equation}
then $\bar{B}_2=\min\{B_{\mu_1\mu_2},B_{\mu_1\mu_3},B_{\mu_{2}\mu_3} \}<\bar{B}_1$, it remains to prove $\hb_3<\bar{B}_2 $.

\vskip 0.1in

We shall prove $\hb_3< B_{\mu_1\mu_2}$. 
\vskip 0.1in

Recall \eqref{tu function}, we note that for any $s\in\R $, there exists a unique $t(s):=t_{u_{12},v_{12},s\phi }$ such that
$ (t(s)u_{12}, t(s)v_{12}, t(s)s\phi ) \in\MR' $ and
\begin{equation*}
	\begin{aligned}
	    t(s)^{2p-2} & \bigg( \int\limits_{\om}(\mu_1u_{12}^{2p}+\mu_2v_{12}^{2p}+\mu_3|s|^{2p}\phi^{2p}\\
	    + & 2\beta_{12} u_{12}^{p}v_{12}^{p}+2\beta_{13}|s|^p u_{12}^{p}\phi^{p}+ 2\beta_{23}|s|^p v_{12}^{p}\phi^{p} ) \bigg) \\
	    = & \int\limits_{\om}(|\nabla u_{12}|^2+\lambda_1u_{12}^2 +|\nabla v_{12}|^2+\lambda_2v_{12}^2 + s^2|\nabla \phi|^2+\lambda_3s^2\phi^2 ),
	\end{aligned}
\end{equation*}
where $(u_{12},v_{12}) $ is a positive least energy solution of \eqref{problem 2} and $\phi\in H^1_0(\om) $ is a positive function.
Recall that
\begin{equation*}
	\begin{aligned}
		NB_{\mu_1\mu_2}
		& =\int\limits_{\om}(\mu_1u_{12}^{2p}
		+2\beta_{12} u_{12}^{p}v_{12}^{p}+\mu_2v_{12}^{2p})\\
		& = \int\limits_{\om}(|\nabla u_{12}|^2+\lambda_1u_{12}^2 +|\nabla v_{12}|^2+\lambda_2v_{12}^2 )~:=c_{12},
	\end{aligned}
\end{equation*}
and we denote
\begin{equation}
	\begin{aligned}
	g(s) & :=t(s)^{2p-2}\\
	& =\frac{ c_{12}+s^2\int_{\om}(|\nabla \phi|^2+\lambda_3\phi^2 ) }{c_{12}+  \int_{\om}(\mu_3|s|^{2p}\phi^{2p}+2\beta_{13}|s|^p u_{12}^{p}\phi^{p}+ 2\beta_{23}|s|^p v_{12}^{p}\phi^{p} ) }
	:=\frac{A(s)}{B(s)}.
	\end{aligned}
\end{equation}
Note that $t'(s)=\frac{1}{2p-2}g(s)^{\frac{3-2p}{2p-2}}g'(s)$ and
\begin{equation}\label{g(0)=1 }
	\lim\limits_{s\to 0}g(s)= g(0)=t(0)=1,
	\quad \lim\limits_{s\to 0}A(s)=\lim\limits_{s\to 0}B(s)=c_{12}.
\end{equation}
By direct computations we have
\begin{equation*}
	\begin{aligned}
		g'(s)=\frac{A'(s)}{B(s)}-\frac{A(s)B'(s) }{B^2(s)} := F_1(s)-F_2(s).
	\end{aligned}
\end{equation*}
Note that $1<p<2$, then 
\begin{equation}\label{F1s}
	\lim_{s\to0} \frac{F_1(s)}{|s|^{p-2}s}=
	\lim_{s\to0}\frac{2s^{2-p}\int(|\nabla \phi|^2+\lambda_3\phi^2 )  }{c_{12}+  \int(\mu_3|s|^{2p}\phi^{2p}+2\beta_{13}|s|^p u_{12}^{p}\phi^{p}+ 2\beta_{23}|s|^p v_{12}^{p}\phi^{p} )}=0.
\end{equation}
On the other hand, we see that
\begin{equation*}
	\begin{aligned}
		F_2(s)=\frac{2p|s|^{p-2}s }{B^2(s)} & \bigg[
		c_{12}+ s^2(\int\limits_{\om}|\nabla \phi|^2+\lambda_3\phi^2)\bigg]\\
		& \cdot\big(|s|^p\int\limits_{\om} \mu_3\phi^{2p}+\int\limits_{\om}\beta_{13}u_{12}^p\phi^p+\beta_{23}v_{12}^p\phi^p\big),
	\end{aligned}
\end{equation*}
then
\begin{equation}\label{F2s}
	\lim_{s\to0}\frac{F_2(s)}{|s|^{p-2}s}
	=\frac{2p}{c_{12}}\int\limits_{\om}\beta_{13}u_{12}^p\phi^p+\beta_{23}v_{12}^p\phi^p.
\end{equation}
Then we have
\begin{equation*}
	\lim_{s\to0}\frac{t'(s)}{|s|^{p-2}s}=-\frac{p}{(p-1)c_{12}}\int\limits_{\om}\beta_{13}u_{12}^p\phi^p+\beta_{23}v_{12}^p\phi^p,
\end{equation*}
that is,
\begin{equation*}
	t'(s)=-\frac{p  }{(p-1)c_{12}}\big(\int\limits_{\om}\beta_{13}u_{12}^p\phi^p+\beta_{23}v_{12}^p\phi^p\big)~|s|^{p-2}s~(1+o(1)),\quad \hbox{as}\;~s\to 0,
\end{equation*}
and so
\begin{equation*}
	t(s)=1-\frac{1  }{(p-1)c_{12}}\big(\int\limits_{\om}\beta_{13}u_{12}^p\phi^p+\beta_{23}v_{12}^p\phi^p\big)~|s|^{p}~(1+o(1)),\quad \hbox{as}\;~s\to 0.
\end{equation*}
This implies that
\begin{equation*}
	t(s)^{2p}=1-\frac{2p  }{(p-1)c_{12}}\big(\int\limits_{\om}\beta_{13}u_{12}^p\phi^p+\beta_{23}v_{12}^p\phi^p\big)~|s|^{p}~(1+o(1)),\quad \hbox{as}\;~s\to 0.
\end{equation*}
Therefore, we deduce from \eqref{maxE} that
\begin{equation*}
	\begin{aligned}
		\hb_3 & \le E(t(s)u_{12},t(s)v_{12},t(s)s\phi )\\
		& = \frac{1}{N}t(s)^{2p}\int\limits_{\om}(\mu_1u_{12}^{2p}+\mu_2v_{12}^{2p}+\mu_3|s|^{2p}\phi^{2p}\\
	    & + 2\beta_{12} u_{12}^{p}v_{12}^{p}+2\beta_{13}|s|^p u_{12}^{p}\phi^{p}+ 2\beta_{23}|s|^p v_{12}^{p}\phi^{p} )\\
	    & = t(s)^{2p}\bigg[B_{\mu_1\mu_2}+\frac{1}{N}\int\limits_{\om}\mu_{3}|s|^{2p} \phi^{2p}+2\beta_{13}|s|^p u_{12}^{p}\phi^{p}+ 2\beta_{23}|s|^p v_{12}^{p}\phi^{p} \bigg]\\
	    & = B_{\mu_1\mu_2}+\frac{1}{N}\int\limits_{\om}\mu_{3}|s|^{2p} \phi^{2p}+2\beta_{13}|s|^p u_{12}^{p}\phi^{p}+ 2\beta_{23}|s|^p v_{12}^{p}\phi^{p}+o(|s|^p)\\
	    & -\frac{1}{N}~\frac{2p}{p-1}|s|^p\int\limits_{\om}\beta_{13}u_{12}^p\phi^p+\beta_{23}v_{12}^p\phi^p\\
	    & -\frac{1}{N}~\frac{2p}{(p-1)c_{12}}|s|^{2p}(\int\limits_{\om}\beta_{13}u_{12}^p\phi^p+\beta_{23}v_{12}^p\phi^p )( \int\limits_{\om}\mu_{3}|s|^p \phi^{2p}+2\beta_{13}u_{12}^{p}\phi^{p}+ 2\beta_{23}v_{12}^{p}\phi^{p} )\\
	    & = B_{\mu_1\mu_2}-(\frac{1}{2}-\frac{1}{N})|s|^p\int\limits_{\om} 2\beta_{13}u_{12}^{p}\phi^{p}+ 2\beta_{23}v_{12}^{p}\phi^{p}
	    +o(|s|^p)\\
	    & <B_{\mu_1\mu_2}  \quad \hbox{as}~~|s|>0~~\hbox{small enough}.
	\end{aligned}
\end{equation*}
That is $\hb_3<B_{\mu_1\mu_2}$. By a similar argument, we can prove that $\hb_3<B_{\mu_i\mu_j}$, then $\hb_3<\bar{B}_2 $, which implies that
	\begin{equation*}
		\hb_3<\min\{\bar{B}_1,\bar{B}_2 \}.
	\end{equation*}
This completes the proof of the case $k=3$.\\

\noindent\textbf{Step 3}\quad ( for general $k\ge 3$ )

Suppose that the $(k-i)$-coupled system \eqref{problem k} with coefficients $(\mu_1,\mu_2,\dots,\mu_{k-i})$ has a positive least energy solution $\textbf{w}_{k-i}=(\omega_1, \omega_2,\dots,\omega_{k-i})$ with energy $B_{\mu_1\dots\mu_{k-i}} $.
Denote
\begin{equation*}
	\WR:=(\omega_1, \omega_2,\dots,\omega_{k-i},s\phi,s^2\phi,\dots,s^{i}\phi),
\end{equation*}
where $\phi \in H^1_0(\om)$ and $ \phi\not\equiv0 $.

Recall \eqref{tu function}, we note that for any $s\in\R $, there exists a unique $t(s):=t_{\WR}$ such that
$ t(s)(\WR) \in\MR'$.
Then by similar argument as in {\textbf{Step 2} }, we have
		
\begin{equation*}
	\begin{aligned}
		\hb_k & \le E(t(s)(\omega_1,\dots,\omega_{k-i},s\phi,s^2\phi,\dots,s^{i}\phi))\\
		& = B_{\mu_1\dots\mu_{k-i}}-(\frac{1}{2}-\frac{1}{N})|s|^p
		\int\limits_{\om}\sum_{j=1}^{k-i}2\beta_{j,k-i+1}|\omega_j|^p|\phi|^p+o(|s|^p)\\
		& <B_{\mu_1\dots\mu_{k-i}},
	\end{aligned}
\end{equation*}
which implies that $\hb_k<\bar{B}_{k-i}$. This completes the proof.
\end{proof}


Let's begin to prove Theorem \ref{main theorem}.

\vskip 0.1in

{\it Proof of Theorem \ref{main theorem}.~~}
Assume that $\beta_{ij}>0$. Since the functional $E$ has a mountain pass structure, by the mountain pass theorem \cite{ambrosetti1973dual,willem1997minimax} there exists $\{\vu_n\}\in H $, such that
\begin{equation*}
	\lim_{n\to+\iy} E(\vu_n)=\hb,\quad \lim_{n\to+\iy} E'(\vu_n)=0,
\end{equation*}
where $\vu_n=(u_{1n},u_{2n},\dots,u_{kn}) $. By standard argument it is easy to see that $\{\vu_n \} $ is bounded in $H$, and so we may assume that $(u_{1n},\dots,u_{kn}) \rightharpoonup (u_1,\dots,u_k) $ weakly in $H$.
Passing to a subsequence, we may assume that $u_{in}\rightharpoonup u_i $ weakly in $L^{2p}(\om)$, $|u_{in}|^{q-1}u_{in}\rightharpoonup |u_i|^{q-1}u_i$ weakly in $L^{2p/q}(\om)$ if$ ~1<q<2p$, and $u_{in}\to u_i $  strongly in $L^2(\om)$. Since $E'(\vu_n)\to 0$, then we have $E'(u_1,\dots,u_k)=0 $.

Set $\sigma_{in}=u_{in}-u_i$, then by Br{\'{e}}zis-Lieb Lemma, 
there holds
\begin{equation}\label{bl1}
	|u_{in}|^{2p}_{2p}=|u_{i}|^{2p}_{2p}+|\sigma_{in}|^{2p}_{2p}+o(1),\quad i=1,2,\dots,k.
\end{equation}
Note that $\vu_n\in\MR $ and $E'(\vu_n)=0 $. Then combined with Lemma \ref{Brezis-Lieb type lemma }, we have
\begin{equation}\label{bl2}
	\int\limits_{\om}|\nabla\sigma_{in}|^2
	-\int\limits_{\om}(\mu_i|\sigma_{in}|^{2p}+\sum_{j=1,j\ne i}^{k} \beta_{ij} |\sigma_{in}|^{p}|\sigma_{jn}|^{p} )=o(1),
\end{equation}
\begin{equation}\label{bl3}
	E(u_{1n},\dots,u_{kn})=E(u_{1},\dots,u_{k})+I(\sigma_{1n},\dots,\sigma_{kn})+o(1),
\end{equation}
for $i=1,2,\dots,k.$

Passing to a subsequence, we may assume that
\begin{equation*}
	\lim_{n\to+\iy}\int\limits_{\om}|\nabla\sigma_{in}|^2=b_i,\quad ~~i=1,2,\dots,k,
\end{equation*}
then by \eqref{bl2} we have
\begin{equation*}
	I(\sigma_{1n},\dots,\sigma_{kn})=\frac{1}{N}(b_1+\cdots+b_k)+o(1).
\end{equation*}
Letting $n\to +\iy$ in \eqref{bl3}, we get that
\begin{equation}\label{lim E(un)}
	\begin{aligned}
		0\le E(u_1,\dots,u_k)
		& \le  E(u_1,\dots,u_k)+\frac{1}{N}(b_1+\cdots+b_k)\\
		& =\lim_{n\to+\iy}E(\vu_n)=\hb.
	\end{aligned}
\end{equation}

\noindent\textbf{Case 1.} $u_1, u_2,\dots,u_k\equiv 0.$

By \eqref{lim E(un)}, we have $b_1+b_2+\cdots+b_k>0$.
Then we assume that $(\sigma_{1n},\dots,\sigma_{kn})\ne (0,\dots,0)$ for $n$ large. Recall the definition of $\NR'$, and by \eqref{bl2}, it is easy to check that there exists $t_n $ such that $(t_n\sigma_{1n},\dots,t_n\sigma_{kn})\in \NR' $ and $t_n\to 1$ as $n\to+\iy$.
Then by \eqref{A=A'} and \eqref{lim E(un)}, we have
\begin{equation*}
	\begin{aligned}
	    \hb=\frac{1}{N}\sum_{i=1}^{k}b_i
		& =\lim_{n\to+\iy} I(\sigma_{1n},\dots,\sigma_{kn})\\
		& =\lim_{n\to+\iy} I(t_n\sigma_{1n},\dots,t_n\sigma_{kn})\\
		&\ge A'=A,
	\end{aligned}
\end{equation*}
a contradiction with Lemma \ref{key lemma}. Therefore, Case 1 is impossible.
\vskip 0.1in

\noindent\textbf{Case 2.} $u_i\not\equiv0 $ and $u_j\equiv0 $ for $i\in I\subsetneqq\{1,2,\dots,k\} $ and  $j\in \{1,2,\dots,k\}\backslash I $.

Without loss of genrality, we may assume that $u_1,\dots,u_{\tau} \not\equiv0, u_{\tau+1},\dots, u_k\equiv 0$. 
Then $(u_1,\dots,u_{\tau})$ is a nontrivial solution of the $\tau$-coupled system, and so
\begin{equation*}
	\hb \ge E(u_1,\dots,u_{\tau},0,\dots,0) \ge B_{\mu_1\ldots\mu_{\tau}}\ge \bar{B}_{\tau},
\end{equation*}
a contradiction with Lemma \ref{key lemma}. Therefore, Case 2 is also impossible.
\vskip 0.1in

Therefore, $u_1,\dots,u_k\not\equiv0$.

Since $E'(u_1,\dots,u_k)=0$, we have $(u_1,\dots,u_k)\in\MR $, by $\hb\le B $ and \eqref{lim E(un)} we have 
\begin{equation}\label{hb=B}
	E(u_1,\dots,u_k)=\hb=B.
\end{equation}
This means $(|u_1|,\dots,|u_k|)\in \MR\subset\MR' $ and $E(|u_1|,\dots,|u_k|)=\hb=B. $ By \eqref{M' definition} and \eqref{hb E}, there exists a Lagrange multiplier $\gamma\in\R $ such that
\begin{equation*}
 	E'(|u_1|,\dots,|u_k|)-\gamma G'(|u_1|,\dots,|u_k|)=0.
\end{equation*} 
Since 
$ E'(|u_1|,\dots,|u_k|)(|u_1|,\dots,|u_k|)=G(|u_1|,\dots,|u_k|)=0$ and 
\begin{equation*}
	\begin{aligned}
	   & G'(|u_1|,\dots,|u_k|)(|u_1|,\dots,|u_k|)\\
	   & =(2-2p)\int\limits_{\om}(\sum_{i=1}^{k}\mu_i|u_i|^{2p}+\sum_{i=1,j>i}^{k} 2\beta_{ij} |u_{i}|^{p}|u_{j}|^{p} ) < 0,
	\end{aligned}
\end{equation*}
we get that $\gamma=0 $ and $E'(|u_1|,\dots,|u_k|)=0$. This means $(|u_1|,\dots,|u_k|)$ is a least energy solution of \eqref{problem}.
Then by the maximum principle, we see that $|u_1|,\dots,|u_k|>0$ in $\om$.
Therefore, $(|u_1|,\dots,|u_k|)$ is a positive least energy solution of \eqref{problem}.
\hfill$\Box$

\vskip 0.2in

{\it Proof of Theorem \ref{symmetric theorem existence and uniqueness}}
~Assume that $N\ge 5, -\la_1(\Omega)<\la_1=\cdots=\la_k:=\la <0$ and $\mu_i,\beta_{ij}>0$.
The following proof is similar to the proof of Theorem \ref{main theorem_limit} (1).
It is well known that the Br{\'{e}}zis-Nirenberg problem \cite{Brezis1983}
\begin{equation*}
	-\Delta u+\lambda u=|u|^{2^*-2}u, \quad u\in H^1_0(\om) 	
\end{equation*}
has a positive least energy solution $\omega $ that attains 
\begin{equation*}
 	S_\lambda:=\inf\limits_{u\in H^1_0(\Omega)\backslash\{0\} }\dfrac{\int (|\nabla u|^2+\lambda u^2)~\mathrm{d}x }{(\int |u|^{2^*}~\mathrm{d}x )^\frac{2}{2^*} }.
\end{equation*}
By \eqref{hb E} and \eqref{hb=B}, we have
	\begin{equation*}
		B:=\inf\limits_{\vu\in \MR}E(\vu)
		=\inf\limits_{\vu\in \MR'}E(\vu),
	\end{equation*}
where
\begin{equation*}
	\begin{aligned}
	    E(\vu) 
	    =\dfrac{1}{2}\int\limits_{\om}\sum_{i=1}^{k}(|\nabla u_i|^2+\lambda u_i^2)
	    -\dfrac{1}{2p}\int\limits_{\om}(\sum_{i=1}^{k}\mu_i|u_i|^{2p}+\sum_{i=1,j>i}^{k} 2\beta_{ij} |u_{i}|^{p}|u_{j}|^{p} ),
	\end{aligned}
\end{equation*}
\begin{equation*}
		\MR=\{\vu\in H ~|~u_i\not\equiv 0,~E'(\vu)(0,\dots,u_i,\dots,0)=0~\hbox{for}\; i=1,2,\dots,k \},
\end{equation*}
\begin{equation*}
		\MR'=\{\vu\in H\backslash\{\vec{\textbf{0}}\}~|~E'(\vu)~\vu=0~\hbox{for}\; i=1,2,\dots,k \}.
\end{equation*}
Then by standard argument(cf. \cite{peng2016elliptic}), we have
\begin{equation}\label{B lambda}
	B=\inf_{\vu\in H\backslash\{\vec{\textbf{0}}\}}
	\dfrac{ \big(\sum_{i=1}^{k}\int(|\nabla u_i|^2+\lambda u_i^2)~ \big)^{\frac{N}{2}}}
	{N\big(\sum_{i=1}^{k}\mu_i\int|u_i|^{2p}+\sum_{i,j=1,i< j}^{k} 2\beta_{ij}\int|u_{i}u_{j}|^{p} \big)^{\frac{N-2}{2} } }.
\end{equation}
Recall the minimizing problem introducd in the proof of Lemma \ref{existence algebric}:
	\begin{equation*}
		d_k:=\inf\limits_{\vt\in\PR_k }\GR_k(\vt),
	\end{equation*}
	where
	\begin{equation*}
		\GR_k(\vt)=\sum\limits_{i=1}^k(\frac{t_i^2}{2}-\frac{\mu_i|t_i|^{2p}}{2p})- \frac{1}{2p}\big(\sum\limits_{i=1,j>i}^{k}2\beta_{ij}|t_i|^{p}|t_j|^{p} \big),
	\end{equation*}
	\begin{equation*}
		\PR_k= \{\vt\in\R^k \backslash\{\vec{\textbf{0} } \} ~|~P_k(\vt):=\sum\limits_{i=1}^k(t_i^2-\mu_i|t_i|^{2p})-  \sum\limits_{i=1,j>i}^{k}2\beta_{ij}|t_i|^{p}|t_j|^{p}=0   \}.
	\end{equation*}
By standard argument (cf. \cite{peng2016elliptic}), we see that
\begin{equation*}
	d_k=\inf_{\vt\in\R^k\backslash\{\vec{\textbf{0}}\}}
	\dfrac{ \big(\sum_{i=1}^k t_i^2\big)^{\frac{N}{2}}}{N\big(\sum\limits_{i=1}^k\mu_i|t_i|^{2p}-  \sum\limits_{i=1,j>i}^{k}2\beta_{ij}|t_i|^{p}|t_j|^{p} \big)^{\frac{N-2}{2} } },
\end{equation*}
then by Sobolev inequality
\begin{equation*}
	\int\limits_\Omega (|\nabla u|^2+\lambda u^2)~\mathrm{d}x
	\ge S_\lambda\bigg(\int\limits_\Omega |u|^{2p}~\mathrm{d}x\bigg)^\frac{2}{2p},
\end{equation*}
we have
\begin{equation*}
	B\ge \inf_{\vu\in H\backslash\{\vec{\textbf{0}}\}}
	\dfrac{ \big(\sum_{i=1}^{k}\Vert u_i\Vert_{2p}^2~ \big)^{\frac{N}{2}}}
	{N\big(\sum_{i=1}^{k}\mu_i\Vert u_i\Vert^{2p}_{2p}+\sum_{i,j=1,i< j}^{k} 2\beta_{ij} \Vert u_{i}\Vert^p_{2p} \Vert u_{j}\Vert^{p}_{2p} \big)^{\frac{N-2}{2} } }S_{\lambda}^{\frac{N}{2}},
\end{equation*}
where $\Vert\cdot\Vert_{q}$ denotes the norm on $L^{q}(\Omega)$ and $2p=2^*=\frac{2N}{N-2} $.
Recall Lemma \ref{existence algebric}, then we see that 
\begin{equation*}
	B=d_kS_\lambda^\frac{N}{2}
\end{equation*}
is attained by
\begin{equation*}
	(t_1\omega,t_2\omega,\dots,t_k\omega),
\end{equation*}
where $(t_1, t_2, \dots, t_k) $  is the positive solution of \eqref{algebric system} that attains $d_k$.
In fact, we deduce that
\begin{equation}\label{minimizing B}
	B= \inf_{\vu\in D\backslash\{\vec{\textbf{0}}\}}
	\dfrac{ \big(\sum_{i=1}^{k}\Vert u_i\Vert_{2p}^2~ \big)^{\frac{N}{2}}}
	{N\big(\sum_{i=1}^{k}\mu_i\Vert u_i\Vert^{2p}_{2p}+\sum_{i,j=1,i< j}^{k} 2\beta_{ij} \Vert u_{i}\Vert^p_{2p} \Vert u_{j}\Vert^{p}_{2p} \big)^{\frac{N-2}{2} } }S_{\lambda}^{\frac{N}{2}}.
\end{equation}
Suppose that $B$ is attained  by some nonzero $\vec{\textbf{v}}=(v_1,v_2,\dots,v_k)$.
Then by H\"older inequality and Sobolev inequality, we see from \eqref{B lambda} and \eqref{minimizing B} that 
 \begin{equation*}
 	 \int\limits_\Omega(|\nabla v_i|^2+\lambda v_i^2)~\mathrm{d}x =S_\lambda\Big(\int\limits_\Omega |v_i|^{2p}~\mathrm{d}x \Big)^{\frac{2}{2p}},
 \end{equation*}	 
 \begin{equation*}
 	 \quad \Big(\int\limits_\Omega |v_iv_j|^{p}~\mathrm{d}x\Big)^2=\int\limits_\Omega |v_i|^{2p}~\mathrm{d}x\int\limits_\Omega |v_j|^{2p}~\mathrm{d}x
 \end{equation*}
 for all $i,j=1,2,\dots,k$, which implies $v_i=0$ or $v_i=c_i\omega_{i}$ for some $c_i\ne0$, where $\omega_i$ is a positive least energy solution of Br\'ezis-Nirenberg problem \eqref{Brezis-Nirenberg problem 1}.

Moreover, for $v_{i}=c_i\omega_i\ne 0$, H\"older's inequality implies that there exists $k_i, k_j>0$, such that $k_i(c_iw_i)^2=k_j(c_jw_j)^2$.
Thus we can redenote all $v_i=c_i\omega $ for all $i$ where $c_{i}=0 $ or $c_{i}\ne0$.

By \eqref{minimizing B}, we have that $ (\Vert v_1\Vert_{2p}, \Vert v_2\Vert_{2p}, \dots, \Vert v_k\Vert_{2p} )$ attains $d_k$. 
It is easy to see that $(c_1, c_2,\dots, c_k) $ attains $d_k$.
Recall Lemma \ref{existence algebric} that $d_k $ is attained by a positive solution if $\beta_{ij}>0$. 
Then $c_i>0$ for all $i=1,2,\dots,k$.

That is, $B$ is attained by $\vu$ if and only if
\begin{equation*}
	\vu=(\tilde{t}_1\omega,\tilde{t}_2\omega,\dots,\tilde{t}_k\omega),
\end{equation*}
where $ (\tilde{t}_1, \tilde{t}_2, \dots, \tilde{t}_k)$ is defined in \eqref{least algebric solution}.
Then the positive least energy solution of the system \eqref{problem} must be of the synchronized form $(\tilde{t}_1\omega,\tilde{t}_2\omega,\dots,\tilde{t}_k\omega)$.

Now we assume that $\beta_{ij}:=\beta>0, i\ne j, i,j=1,2,\dots,k$.
Assume that $\Omega\subset\RN$ is a ball, then the positive least energy solution of the Br{\'{e}}zis-Nirenberg problem \eqref{Brezis-Nirenberg problem 1}
is unique (cf. \cite{Brezis1983}).
Recall \cite[Proposition 3.3]{Wu2019a} that there exists $\beta_k>0$, such that the algebric system \eqref{algebric system} has a unique solution that attains $d_{k}$ for $\beta>\beta_k$.
Therefore, the positive least energy solution of \eqref{problem} is unique under these assumptions.
\hfill$\Box$

\vskip 0.2in

%

\noindent{\bf Acknowledgements}  The authors wish to thank  Dr. Y. Wu  very much for his  kind suggestion  and  valuable comments on the proof of Theorem  \ref{symmetric theorem existence and uniqueness}.


\vskip0.26in


\end{document}